\documentclass[a4paper,reqno,11pt]{amsart}
\usepackage{amsfonts,amsmath,amsthm,amssymb,stmaryrd}
\usepackage{hyperref}
\usepackage{color}
\usepackage{mathrsfs}
\newtheorem{theorem}{Theorem}[section]
\newtheorem{lemma}{Lemma}[section]
\newtheorem{proposition}[lemma]{Proposition}
\newtheorem{remark}{Remark}[section]
\newtheorem{lemma(Gagliardo-Nirenberg)}{Lemma{(Gagliardo-Nirenberg)}}[section]
\numberwithin{equation}{section}
\allowdisplaybreaks
\begin{document}
\title[Global Classical Solutions to the viscous two--phase flow model]{Global Classical Solutions to the viscous two--phase flow model with slip Boundary Conditions in 3D Exterior Domains}
\author{Zilai Li}
\address{School of Mathematics and Information Science, Henan Polytechnic University, Jiaozuo, Henan 454003, China}
\email[Z.L. Li]{lizl@hpu.edu.cn}
\author{Hao Liu}
\address{School of Mathematics and Information Science, Henan Polytechnic University, Jiaozuo, Henan 454003, China}
\email[H. Liu]{1131987412@qq.com}
\author{Huaqiao Wang}
\address{College of Mathematics and Statistics, Chongqing University,
                             Chongqing, 401331, China.}
\email[H.Q. Wang]{wanghuaqiao@cqu.edu.cn}

\date{\today}

\begin{abstract}
We consider the  two--phase flow model with slip boundary condition in a 3D exterior domains whose boundary is smooth. We establish the global existence of classical solutions of this system provided that the initial energy is suitably small. Moreover, the pressure has large oscillations and contains vacuum states when the initial pressure allows large oscillations and a vacuum. Finally, we also obtain the large--time behavior of the classical solutions.
\end{abstract}

\maketitle

\section{Introduction}
The two--phase flow model  originally developed by Zuber and Findlay \cite{Z1964}, Wallis \cite{W1979}, and Ishii \cite{I1975,ITT2003} can be written as
\begin{equation}\label{1.1}
\left\{
	\begin{array}{lr}
	\rho_t+{\rm div}(\rho u)=0,&\\
	m_t+{\rm div}(mu)=0,&\\
	\left((\rho+m) u\right)_t+{\rm div}{[(\rho+m)u\otimes u]}+\nabla P(\rho,m)=\mu\Delta u+(\lambda+\mu)\nabla {\rm div}u,
	\end{array}
\right.
\end{equation}
which is commonly used in industrial applications, such as nuclear, power, chemical--process, oil and gas, cryogenics, bio--medical, micro--technology and so on. Here $(x,t)\in \Omega \times(0, T]$, $\Omega$ is a domain in\ $\mathbb{R}^3$. 
$\rho\ge0$, $m\ge0$, $u=(u_1,u_2,u_3)$\ and\ $ P(\rho,m)=\rho^\gamma+m^\alpha~(\gamma>1,\alpha\ge1)$ are the unknown two--phase flow model's fluid density, velocity and pressure, respectively.  The constants $\mu$ and $\lambda$ are the shear viscosity and bulk coefficients respectively satisfying the following physical restrictions:
 \begin{equation}\label{1.2}
	\mu>0,\ 2\mu+3\lambda\ge0.
 \end{equation}

In this paper, we consider the domain $\Omega$ is the exterior of a simply connected bounded domain D in $\mathbb{R}^3$, and its boundary $\partial\Omega$ is smooth. In addition, the system is studied subject to the given initial data
\begin{equation}\label{1.3}
	\rho(x,0)=\rho_0(x),\ m(x,0)=m_0(x),\ (\rho+m)u(x,0)=(\rho_0+m_0)u_0(x),\ x\in\Omega,
\end{equation}
and slip boundary condition
\begin{equation}\label{1.4}
	u\cdot n=0,\ {\rm curl}u\times n=0\ {\rm on}\ \partial\Omega,
\end{equation}
with the far field behavior
\begin{equation}\label{1.5}
	u(x,t)\to0,\ \rho(x,t)\to\rho_{\infty}\ge0,\ m(x,t)\to m_{\infty}\ge0,\text{ as }\ |x|\to\infty,
\end{equation}
where $n=(n^1,n^2,n^3)$ is the unit outward normal vector to the boundary $\partial\Omega$ pointing outside $\Omega$, $\rho_\infty$ and $m_\infty$ are the non--negative constants.

The first condition in \eqref{1.4} is the non--penetration boundary condition, while the second one is also known in the form
\begin{equation*}
	\left(D(u)n\right)_{\tau}=-\kappa_\tau u_\tau,
\end{equation*}
where $D(u)=(\nabla u+(\nabla u)^{tr})/2$ is the deformation tensor, $\kappa_\tau$ is the corresponding normal curvature of $\partial\Omega$ in the $\tau$ direction and the
symbol $u_\tau$ represents the projection of tangent plane of the vector $u$ on $\partial\Omega$
.  This type of boundary condition was originally introduced by Navier \cite{N1827} in 1823, which was followed by many applications, numerical studies and analysis for various fluid mechanical problem, see, for instance \cite{CF1988,ITT2003,S1959} and the references therein.

 Many models are related to the two--phase model \eqref{1.1}, specially,  the case of $\alpha=1$ corresponds to the hydrodynamic equations which was derived as the asymptotic limit of  Vlasov--Fokker--Planck equations coupled with compressible Navier--Stokes equations, see  \cite{CG2006,MV2008}; The case of $\alpha=2$ is associated with a compressible Oldroyd--B type model with stress diffusion, see \cite{BLS2017}. Furthermore, if we let $m\equiv0$, then the viscous liquid--gas two--flow model \eqref{1.1} reduces to the classical isentropic compressible  Navier--Stokes equations.  Comparing with the isentropic compressible Navier--Stokes equations, the main difference is that the pressure law $P(\rho,m)=\rho^{\gamma}+m^{\alpha}$ depends on two different variables from the continuity equations.

Before stating our main result, we briefly recall some previous known results on the viscous two--fluid model. For the one--dimensional case, Evje and Karlsen \cite{EK2008} obtained the first global existence result on
weak solutions with large initial data subject to the domination conditions. Later, the domination condition was removed  by Evje--Wen--Zhu \cite{EWZ2017} using the decomposition of the pressure term, which allows transition to each single-phase flow. Recently, Gao--Guo--Li \cite{GGL2022} considered  the Cauchy problem of 1D viscous two--fluid model and established the global existence of strong solutions with the large initial value and vacuum. For more related results, please refer to \cite{E2011,EK2008,YZ2009,YZ2011} and the references therein. For the multi--dimensional case, Yao--Zhang--Zhu \cite{YZZ2010}  proved the global existence of weak solutions to the 2D Cauchy problem case when the initial energy is small and both of the initial densities are positive.  Hao and Li \cite{HL2011} obtained the existence and uniqueness of the global strong solutions to the Cauchy problem in $\mathbb{R}^d$ with $d\geq 2$ in the framework of Besov spaces, where the possible vacuum state is included in the equilibrium state for the gas component at far field. Zhang and Zhu \cite{ZZ2015} considered the 3D Cauchy problem and proved the global existence of a strong solution when $H^2$--norm of the initial perturbation around a constant state is sufficiently small. When both phases contain vacuum initially, Guo--Yang--Yao \cite{GYY2011} proved the global existence of  strong solutions to the 3D Cauchy problem under the assumption that initial energy is sufficiently small. Very recently, the domination condition was removed by Yu \cite{Y2021} for the global existence of the strong solution to the 3D case when the initial energy is small. For large initial data cases,
Vasseur--Wen--Yu \cite{VWY2019} obtained the global existence of weak solutions to Dirichlet boundary value problem of  \eqref{1.1}  in $\mathbb{R}^3$ with the pressure  $ P(\rho,m)=\rho^r+m^\alpha~(r>1,\alpha\ge1)$ and the domination conditions. Novotn\'{y} and Pokorn\'{y} \cite{NP2020} extended the domination condition to the case that both $\gamma$ and $\alpha$ can touch $\frac{9}{5}$, where more general pressure laws covering the cases of
$P(\rho,m)=\rho^r+m^\alpha~(r>1,\alpha\ge1)$ were considered.  Wen \cite{W2019} obtained the global existence of weak solutions to 3D Dirichlet problem of compressible two--fluid model without any domination conditions. However, there are few results about classical solutions to compressible two--fluid model for general bounded domains, which  is one of our main motivations of present paper.

When we take $m=0$ in \eqref{1.1}, then two--phase flow model \eqref{1.1} changes into the compressible Navier--Stokes equations. In the last several decades, significant progress on the compressible Navier--Stokes equations has been achieved by many authors in the analysis of the well--posedness and large time behavior. We only briefly review some results related to the existence of strong or classical solutions.
The global classical solutions were first obtained by Matsumura--Nishida \cite{MN1980} for initial
data close to a nonvacuum equilibrium in $H^{3} (\mathbb R^3)$. It is worth mentioned that their results
have been improved by Huang--Li--Xin \cite{HLX2012} and Li--Xin \cite{LX2018}, in which the global existence
of classical solutions is obtained with smooth initial data that are of small energy but
possibly large oscillations. Very recently, for the 3D bounded domain (or 3D Exterior Domains) with slip boundary conditions, Cai--Li \cite{CL2021}(or Cai--Li--Lv \cite{CLL2021}) proved that the existence and large--time behavior of global classical solutions to the compressible Navier--Stokes equations. And Cai--Huang--Shi \cite{CHS2021} proved  the global existence and exponential growth of classical solutions subject to large potential forces with slip boundary condition in 3D bounded domains. For 3D bounded Domains with Non--Slip Boundary Conditions, Fan--Li \cite{FL2022} proved global classical solutions to the compressible Navier--Stokes system with vacuum.

Before stating the main results, we explain some notations and conventions used in this paper. We denote
\begin{equation*}\int fdx \triangleq\int_\Omega fdx.\end{equation*}
For integer $k$ and $1\le q <+\infty$, the standard homogeneous Sobolev spaces are denoted as follows:
$$D^{k,q}_0(\Omega) \triangleq
\left\{
{u\in L^1_{loc}(\Omega)\big|\|\nabla^ku\|_{L^q(\Omega)}<+\infty}
\right\},\ \|\nabla u\|_{D^{k,q}(\Omega)}\triangleq\|\nabla^ku\|_{L^q(\Omega)}.$$
We also denote
\begin{equation*}
	D^k(\Omega)=D^{k,2}(\Omega),\ H^k(\Omega)=W^{k,2}(\Omega),\ W^{k,q}(\Omega)=L^q(\Omega)\cap D^{k,p}(\Omega)
\end{equation*}
with the norm $\|u\|_{W^{k,q}(\Omega)}\triangleq\left(\sum_{|m|\le k}\|\nabla^mu\|_{L^q(\Omega)}^q\right)^{\frac{1}{q}}$.

Simply, $L^q(\Omega),\ D^{k,q}(\Omega),\ D^k(\Omega),\ W^{k,q}(\Omega)$ and $H^k(\Omega)$ can be denoted by $L^q,\ D^{k,q},\ D^k,\ W^{k,q}$ and $H^k$ respectively, and set that
\begin{equation*}
	B_{R}\triangleq\left\{x\in \mathbb{R}^3\big||x|<R\right\}.
\end{equation*}

For two $3\times3$ matrices $A=\left\{a_{ij}\right\}$, $B=\left\{b_{ij}\right\}$, the symbol $A:B$ represents the trace of $AB$, $$A:B\triangleq \sum_{i,j=1}^{3}a_{ij}b_{ji}.$$
\par Set the initial total energy of \eqref{1.1}:
\begin{equation}\label{1.6}
C_0\triangleq \int_\Omega\left(\frac{1}{2}(\rho_0+m_0)|u_0|^2+G(\rho_0,m_0)\right)dx
\end{equation}
with
\begin{equation*}
	G(\rho,m)=\rho\int^\rho _{\rho_\infty}\frac{P(s,m)-P(\rho_\infty,m)}{s^2}ds+m\int^m _{m_\infty}\frac{P({\rho},s)-P({\rho},m_\infty)}{s^2}ds.
\end{equation*}
\par Finally, for $v=(v^1,v^2,v^3)$, we set $\nabla_jv=(\partial_jv^1,\partial_jv^2,\partial_jv^3)$, for $j=1,2,3$, $P_0=P(\rho_0,m_0)$, and $P_\infty=P(\rho_\infty,m_\infty)$.

Our first result is stated below:

\begin{theorem}\label{Thm1.1}
Let $\Omega$ be the exterior of a simply connected bounded domain in $\mathbb{R}^3$ with smooth boundary $\partial\Omega$. For 
$M\ge1$, $\bar{\rho}\ge\rho_\infty+1$, $\bar{m}\ge m_{\infty}+1$ and some $q\in(3,6)$, assume that the initial date $(\rho_0,m_0,u_0)$ satisfy the following condition:
\begin{equation}\label{1.7}
	u_0\in\left\{\bar{f}\in D^1\cap D^2\ :\ \bar{f}\cdot n=0,\ {\rm curl}\bar{f}\times n=0\  on\  \partial\Omega\right\},
\end{equation}
\begin{equation}\label{1.8}
	(\rho_0-\rho_\infty,m_0-m_\infty,P(\rho_0,m_0)-P(\rho_\infty,m_\infty))\in H^2\cap W^{2,q},
\end{equation}
\begin{equation}\label{1.9}
	0\le\rho_0\le\bar{\rho},\ \ \ 0\le m_0\le\bar{m},\ \ \ \mu\|{\rm curl}u_0\|_{L^2}^2+(\lambda+2\mu)\|{\rm div}u_0\|_{L^2}^2\triangleq M,
\end{equation}
and the compatibility condition
\begin{equation}\label{1.10}
	-\mu\Delta u_0-(\lambda+\mu)\nabla {\rm div}u_0+\nabla P_0={(\rho_0+m_0)}^{1/2}g,
\end{equation}
for some $g\in L^2$. Then there exists a positive constant $\varepsilon$ depending only on $\lambda$, $\mu$, $\gamma$, $\alpha$, $\Omega$, M, $\bar{\rho}$ and $\bar{m}$ such that if
\begin{equation}\label{1.11}
C_0\le\varepsilon,
\end{equation}
then the slip problem \eqref{1.1}--\eqref{1.5} has a unique global classical solution $(\rho,m,u)$ in $\Omega\times(0,\infty)$ satisfying
\begin{equation}\label{1.12}
	0\le\rho(x,t)\le2\bar{\rho},\ 0\le m(x,t)\le2\bar{m}, 
\end{equation}
\makeatletter
\@addtoreset{equation}{section}
\makeatother
\renewcommand{\theequation}
{\arabic{section}.\arabic{equation}}
\begin{equation}\label{1.13}
	\left\{
	\begin{array}{lr}
	(\rho-\rho_\infty,m-m_\infty,P-P_\infty)\in C([0,\infty);H^2\cap W^{2,q}), & \\
	\nabla u\in C([0,\infty);H^1)\cap L^\infty_{loc}(0,\infty;W^{2,q}), & \\
	u_t\in L^\infty_{loc}(0,\infty;D^1\cap D^2)\cap H^1_{loc}(0,\infty;D^1), &\\
	(\rho+m)^{\frac{1}{2}}u_t\in L^\infty(0,\infty;L^2). &
	\end{array}
	\right.
\end{equation}
In addition, the following large-time behavior
\begin{equation}\label{1.14}
	\mathop{\rm lim}_{t\to\infty}\int\left(|P-P_\infty|^q+(\rho+m)^{\frac{1}{2}}|u|^4+|\nabla u|^2\right)(x,t)dx=0
\end{equation}
holds for any $2< q<\infty$.
\end{theorem}

With \eqref{1.14} at hand, we are able to obtain the following large--time behavior of the gradient of the pressure when vacuum states appear initially. It was just a parrel result which was first established
by Li and his collaborators in \cite{CLL2021}.

\begin{theorem}\label{Thm1.2}
	Under the conditions of Theorem \ref{Thm1.1}, assume further that $P_\infty>0$ and there exists some point $x_0\in\Omega$ such that $P_0(x_0)=0$. Then the unique global classical solution $(\rho,m,u)$ to the problem \eqref{1.1}--\eqref{1.5} obtained in Theorem \ref{Thm1.1} has to blow up as $t\to\infty$ in the sense that for any $3< r<\infty$,
	\begin{equation}\label{1.15}
		\mathop{\rm lim}_{t\to\infty}\|\nabla P(\cdot,t)\|_{L^r}=\infty.
	\end{equation}
\end{theorem}
\begin{remark}
	Since $q>3$, it follows from Sobolev's inequality and $\eqref{1.13}_1$ that
	\begin{equation}\label{1.16}
		\rho-\rho_\infty,\ \ \nabla\rho,\ \ m-m_\infty,\ \ \nabla m\in C\left(\bar{\Omega}\times[0,T]\right).
	\end{equation}
Moreover, it also follows from $\eqref{1.13}_2$ and $\eqref{1.13}_3$ that
\begin{equation}\label{1.17}
	u,\nabla u,\nabla^2 u,u_t\in C\left(\bar{\Omega}\times[\tau,T]\right),
\end{equation}
due to the following simple fact that
\begin{equation*}
	L^2(\tau,T;H^1)\cap H^1(\tau,T;H^{-1})\hookrightarrow C([\tau,T];L^2).
\end{equation*}
Finally, by $\eqref{1.1}_1$ and $\eqref{1.1}_2$, we have
\begin{equation*}
	\rho_t=-u\cdot\nabla\rho-\rho{\rm div}u\in C\left(\bar{\Omega}\times[\tau,T]\right),
\end{equation*}
\begin{equation*}
	m_t=-u\cdot\nabla m-m{\rm div}u\in C\left(\bar{\Omega}\times[\tau,T]\right),
\end{equation*}
which together with \eqref{1.16} and \eqref{1.17} shows that the solution obtained by Theorem \ref{Thm1.1} is a classical one.
\end{remark}
\begin{remark}
	When $\alpha\le1$ and $\gamma>1$, it is easy to show that $\exists0<C_i<1~(i=1,2)$ depends on $\bar{\rho}$, $\bar{m}$, $\rho_\infty$ and $m_\infty$, the following formula holds
	\begin{equation*}
		C_1(m-m_\infty)^2+C_2(\rho-\rho_\infty)^2\le m\int_{m_\infty}^{m}\frac{s-m_\infty}{s^2}ds+\rho\int_{\rho_\infty}^{\rho}\frac{s-\rho_\infty}{s^2}ds.
	\end{equation*}
\end{remark}

Now, we give some comments on the analysis of this paper. Compared with the bounded domains, because the domain is unbounded, there are two difficulties that we have overcome.
First, thanks to \cite{VW1992}(see Lemma $\ref{Lm2.6}$), we can control $\nabla u$ by means of ${\rm div}u$ and ${\rm curl}u$, the other one is how to control the boundary integrals, especially (see \eqref{3.22}),
\begin{equation*} -\int_{\partial\Omega}\sigma^hFu\cdot(\nabla n+(\nabla n)^{tr})^i(u^{\bot}\times n\cdot\nabla u^i)ds.
\end{equation*}
In fact, thanks to
$$\nabla\cdot(g\times h)=\nabla\times g\cdot h-\nabla\times h\cdot g, \nabla\times(\nabla g)=0$$ and divergence theorem, we can control it.

Next, denote
\begin{equation}\label{1.18}
	\dot{v}\triangleq v_t+u\cdot\nabla v,
\end{equation}
and
\begin{equation}\label{1.19}
	F\triangleq(\lambda+2\mu){\rm div}u-(P-{P}_\infty),
\end{equation}
are the material derivative of $v$ and the effective viscous flux, respectively.
Then the equation $\eqref{1.1}_3$ can be written as
\begin{equation}\label{1.20}
	(\rho+m)\dot{u}=\nabla F-\mu\nabla\times{\rm curl}u,
\end{equation}
which together with the boundary condition \eqref{1.4} implies that one can treat $\eqref{1.1}_3$ as a Helmholtz--Wyle decomposition of $(\rho+m)\dot{u}$ which makes it possible to estimate $\nabla F$ and $\nabla{\rm curl}u$. Finally, whereas $u\cdot n=0$ on $\partial\Omega$, we have
\begin{equation}\label{1.21}
	u\cdot\nabla u\cdot n=-u\cdot\nabla n\cdot u,
\end{equation}
which together with ${\rm curl}u\times n=0$ on $\partial\Omega$ is the key to estimating the integrals on the boundary $\partial\Omega$ .
\section {Preliminaries}\label{S2}
\par In this section, we will collect some known facts and elementary inequalities which will be used frequently later. Firstly, we can get the local existence of strong and classical solutions, its proof is similar to \cite{H2021}.
\begin{lemma}\label{Lm2.1}Assume $\Omega$ satisfies the condition of Theorem \ref{Thm1.1} and $(\rho_0,m_0,u_0)$ satisfies \eqref{1.7}, \eqref{1.8} and \eqref{1.10}. Then there exists a small time $T_0>0$ and a unique strong solution $(\rho,m,u)$ to the problem \eqref{1.1}--\eqref{1.5} on $\Omega\times(0,T_0]$ satisfying for any $\tau\in(0,T_0)$,
\[\begin{cases}
	(\rho-\rho_\infty,m-m_\infty,P-P_\infty)\in C([0,\infty);H^2\cap W^{2,q}),\\
	u\in C([0,\infty);D^1\cap D^2),\ \nabla u\in L^2(0,T;H^2)\cap L^{p_0}(0,T;W^{2,q}),\\
	\nabla u\in L^{\infty}(\tau,T;H^2\cap W^{2,q}),\\
	u_t\in L^\infty(\tau,T;D^1\cap D^2)\cap H^1(\tau,T;D^1),\\
	\sqrt{\rho+m}u_t\in L^\infty(0,T;L^2),
\end{cases}\]
where $q\in(3,6)$ and $p_0=\frac{9q-6}{10q-12}\in(1,\frac{7}{6})$.
\end{lemma}
\par Secondly, the following Gagliardo--Nirenberg's inequality (see\cite{N2011}) will be used frequently later.

\begin{lemma}
Let $\Omega$ be the exterior of a simply connected domain $D$ in $\mathbb{R}^3$.  For any $f\in H^1(\Omega)$ and $g\in L^q(\Omega)\cap D^{1,r}(\Omega)$, there exists some generic constants $C>0$ which may depend on p, q and r such that,
\begin{equation}\label{2.1}
	\|f\|_{L^p(\Omega)}\le C\|f\|_{L^2}^\frac{6-p}{2p}\|\nabla f\|_{L^2}^{\frac{3p-6}{2p}},
\end{equation}
\begin{equation}\label{2.2}
	\|g\|_{C(\bar{\Omega})}\le C\|g\|_{L^q}^{q(r-3)/(3r+q(r-3))}\|\nabla g\|_{L^r}^{3r/(3r+q(r-3))},
\end{equation}
for $p\in[2,6]$, $q\in(1,\infty)$, and $r\in(3,\infty)$.
\end{lemma}

Then, in order to get the uniform (in time) upper bound of the density $\rho$ and $m$, we need the following Zlotnik inequality which was first used in Huang--Li--Xin \cite{H2021}.
\begin{lemma}[\cite{Z2000}]\label{Lm2.3}For $g\in C(R)$ and $y,b\in W^{1,1}(0,T)$, assume that the function y satisfies
\begin{align*}
	y'(t)=g(y)+b'(t) \text{ on } [0,T],\,\  y(0)=y^0.
\end{align*}
 If $g(\infty)=-\infty$ and
\begin{equation}\label{2.3}
	b(t_2)-b(t_1)\le N_0+N_1(t_2-t_1)
\end{equation}
for all $0\le t_1<t_2\le T$ with some $N_0\ge0$ and $N_1\ge0$, then
\begin{align*}
	y(t)\le {\rm max}
	\left\{y^0,\hat{\zeta}\right\} +N_0<\infty\ on\ [0,T],
\end{align*}
where $\hat{\zeta}$ is a constant such that
\begin{equation}\label{2.4}
	g(\zeta)\le-N_1\ for\ \zeta\ge\hat{\zeta}.
\end{equation}
\end{lemma}

Next, thanks to \cite{A2014,VW1992}, we have the following two lemmas.
\begin{lemma}\label{Lm2.4}
	Assume a simply connected bounded domain $D\subset\mathbb{R}^3$ with $C^{k+1,1}$ boundary $\partial D$, $1<q<+\infty$ and a integer $k\ge0$, then for $v\in W^{k+1,q}(D)$ with $v\cdot n=0$ on $\partial D$, there exists a constant $C=C(q,k,D)$ such that
	\begin{equation}
		\|v\|_{W^{k+1,q}(D)}\le C\left(\|{\rm div}v\|_{W^{k,q}(D)}+\|{\rm curl}v\|_{W^{k,q}(D)}\right).
	\end{equation}
If $k=0$, it holds that
\begin{equation}
	\|\nabla v\|_{L^q(D)}\le C(\|{\rm div}v\|_{L^q(D)}+\|{\rm curl}v\|_{L^q(D)}).
\end{equation}	
\end{lemma}

\begin{lemma}\label{Lm2.5}
	Assume a bounded domain $D\subset\mathbb{R}^3$ and $C^{k+1,1}$ boundary only has a finite number of two--dimensional connected components. The integer $k\ge0$ and\ $1<q<\infty$, for $v\in W^{k+1,q}(D)$ with $v\times n=0$ on $\partial D$, then exists a positive constant C depending only on $q$, $k$, $\Omega$ such that
	\begin{align*}
		\|v\|_{W^{k+1,q}(D)}\le C(\|{\rm div}v\|_{W^{k,q}(D)}+\|{\rm curl}v\|_{W^{k,q}(D)}+\|v\|_{L^q(D)}).
	\end{align*}
	If D has no holes, then
	\begin{align*}
		\|v\|_{W^{k+1,q}(D)}\le C(\|{\rm div}v\|_{W^{k,q}(D)}+\|{\rm curl}v\|_{W^{k,q}(D)}).
	\end{align*}
\end{lemma}
The following conclusion is shown in \cite{A2014,VW1992}.
\begin{lemma}\label{Lm2.6}
	$\Omega$ is the exterior of D, and D is a simply connected domains in $\mathbb{R}^3$ with $C^{1,1}$ boundary. Then for $v\in D^{1,q}(\Omega)$ with $v\cdot n=0$ on $\partial\Omega$, it holds that
	\begin{equation}\label{2.7}
		\|\nabla v\|_{L^q(\Omega)}\le C(\|{\rm div}v\|_{L^q(\Omega)}+\|{\rm curl}v\|_{L^q(\Omega)})\ for\ any\ 1<q<3,
	\end{equation}
and
\begin{equation}\label{2.8}
	\|\nabla v\|_{L^q(\Omega)}\le C(\|{\rm div}v\|_{L^q(\Omega)}+\|{\rm curl}v\|_{L^q(\Omega)}+\|\nabla v\|_{L^2(\Omega)})\ for\ any\ 3\le q<+\infty.
	\end{equation}
\end{lemma}

Due to \cite{LMR2016}, we obtain the following fact.
\begin{lemma}\label{Lm2.7}
	Suppose that $\Omega$ satisfies the conditions in Lemma \ref{Lm2.6}, for any $v\in W^{1,q}(\Omega)~(1<q<+\infty)$ with $v\times n=0$ on $\partial\Omega$, it holds that
	\begin{equation*}
			\|\nabla v\|_{L^q(\Omega)}\le C(\|v\|_{L^q(\Omega)}+\|{\rm div}v\|_{L^q(\Omega)}+\|{\rm curl}v\|_{L^q(\Omega)}).
	\end{equation*}
\end{lemma}

By Lemmas \ref{Lm2.4}--\ref{Lm2.7}, we further have the following result whose proof is in \cite{CLL2021}.
\begin{lemma}\label{Lm2.8}
	Let $\Omega$ is the exterior of D which a simply connected domain in $\mathbb{R}^3$ with smooth boundary. Such that every $v\in\left\{D^{k+1,p(\Omega)}\cap D^{1,2}(\Omega) | v(x,t)\to0\ as\ |x|\to\infty\right\}$ with $v\cdot n|_{\partial\Omega}=0$ or $v\times n|_{\partial\Omega}=0$, then there exists some positive constant C depending only on $p$, $k$ and $D$ satisfies
	\begin{equation}\label{2.9}
		\|\nabla v\|_{W^{k,q}(\Omega)}\le C(\|{\rm div}v\|_{W^{k,q}(\Omega)}+\|{\rm curl}v\|_{W^{k,q}(\Omega)}+\|\nabla v\|_{L^2(\Omega)}),
	\end{equation}
for any $p\in[2,6]$ and integer $k\ge0$.
\end{lemma}

Then we give the following Beale--Kato--Majda type inequality with respect to the slip boundary condition \eqref{1.4} which was first proved in \cite{BKM1984,K1986} when ${\rm div}u\equiv0$, it can estimate $\|\nabla u\|_{L^\infty}$.

\begin{lemma}[\cite{CLL2021}]\label{Lm2.9}
Assume that $u\cdot n=0$, ${\rm curl}u\times n=0$, $\nabla u\in W^{1,q}$, for $3<q<\infty$, then there exists a constant $C=C(q)$ such that the following estimate holds
\begin{equation}\label{2.10}
	\|\nabla u\|_{L^\infty}\le C(\|{\rm div}u\|_{L^\infty}+\|{\rm curl}u\|_{L^\infty})\ln(e+\|\nabla^2u\|_{L^q})+C\|\nabla u\|_{L^2}+C.
	\end{equation}
\end{lemma}

Finally, we give the following conclusions for F and curl$u$, whose proof is in \cite{CLL2021}. We sketch it here for completeness.
\begin{lemma}
Assume an exterior domain of some simply connected bounded domain $\Omega\subset\mathbb{R}^3$ and its boundary is smooth. For any $2\le p\le6$ and $q\in(1,\infty)$, suppose that $(\rho,m,u)$ is a smooth solution of \eqref{1.1} with the boundary condition \eqref{1.4}, then there exists a positive constant C depending only on p, q, $\lambda$,\ $\mu$ and\ $\Omega$ such that
\begin{equation}\label{2.11}
	\|\nabla F\|_{L^q}\le C\|(\rho+m)\dot{u}\|_{L^q},
\end{equation}
\begin{equation}\label{2.12}
	\|\nabla{\rm curl}u\|_{L^p}\le C(\|(\rho+m)\dot{u}\|_{L^p}+\|(\rho+m)\dot{u}\|_{L^2}+\|\nabla u\|_{L^2}),
\end{equation}
\begin{equation}\label{2.13}
	\|F\|_{L^p}\le C\|(\rho+m)\dot{u}\|_{L^2}^{\frac{3p-6}{2p}}(\|\nabla u\|_{L^2}+\|P-P_{\infty}\|_{L^2})^{\frac{6-p}{2p}}.
\end{equation}
Moreover,
\begin{equation}\label{2.14}
	\|{\rm curl}u\|_{L^p}\le C\|(\rho+m)\dot{u}\|_{L^2}^{\frac{3p-6}{2p}}\|\nabla u\|_{L^2}^{\frac{6-p}{2p}}+C\|\nabla u\|_{L^2},
\end{equation}
\begin{equation}\label{2.15}
	\|\nabla u\|_{L^p}\le C(\|(\rho+m)\dot{u}\|_{L^2}+\|P-P_\infty\|_{L^6})^{\frac{3p-6}{2p}}\|\nabla u\|_{L^2}^{\frac{6-p}{2p}}+C\|\nabla u\|_{L^2}.
\end{equation}
\end{lemma}
\begin{proof}First, due to $\eqref{1.1}_3$, it is easy to find that $F$ satisfies
\begin{equation*}
	\int\nabla F\cdot\nabla\eta dx=\int(\rho+m)\dot{u}\cdot\nabla\eta dx,\ \ \forall\eta\in C_0^\infty(\mathbb{R}^3)
\end{equation*}
i.e.,
\begin{align*}
\begin{cases}
		-\triangle F={\rm div}((\rho+m)\dot{u}),\ &\text{ in } \Omega, \\
		\frac{\partial F}{\partial n}=-((\rho+m)\dot{u})\cdot n,\ &\text{ on } \partial\Omega, \\
		\nabla F\to0,\ &\text{ as }\ |x|\to\infty.
\end{cases}
\end{align*}
It follows from \cite{NS2004} that for some $q\in(1,\infty)$,
\begin{equation}\label{2.16}
	\|\nabla F\|_{L^q}\le C\|(\rho+m)\dot{u}\|_{L^q},
\end{equation}
and
\begin{equation}\label{2.17}
	\|\nabla F\|_{W^{k,q}}\le C\|(\rho+m)\dot{u}\|_{W^{k,q}},
\end{equation}
for any integer $k\ge1$.

Due to \eqref{1.20} and \eqref{1.4}, from Lemma \ref{Lm2.7} and \eqref{2.16}, we get
\begin{equation}\label{2.18}
	\begin{aligned}
		\|\nabla{\rm curl}u\|_{L^q}&\le C\|{\rm curl}u\|_{L^q}+\|{\rm div}{\rm curl}u\|_{L^q}+\|\nabla\times{\rm curl}u\|_{L^q}\\
		&\le C(\|{\rm curl}u\|_{L^q}+\|(\rho+m)\dot{u}\|_{L^q}+\|\nabla F\|_{L^q})\\
		&\le C(\|(\rho+m)\dot{u}\|_{L^q}+\|{\rm curl}u\|_{L^q}).
	\end{aligned}
\end{equation}

By virtue of Lemma \ref{Lm2.8}, \eqref{1.20}, \eqref{2.17} and \eqref{2.18}, it indicates that
\begin{equation}\label{2.19}
	\begin{aligned}
		\|\nabla{\rm curl}u\|_{W^{k,q}}&\le C\|{\rm div}{\rm curl}u\|_{W^{k,q}}+\|{\rm curl}{\rm curl}u\|_{W^{k,q}}+\|\nabla{\rm curl}u\|_{L^2}\\
		&\le C(\|\nabla F\|_{W^{k,q}}+\|(\rho+m)\dot{u}\|_{W^{k,q}}+\|\rm{curl u}\|_{L^2}+\|(\rho+m)\dot{u}\|_{L^2})\\
		&\le C(\|(\rho+m)\dot{u}\|_{W^{k,q}}+\|(\rho+m)\dot{u}\|_{L^{2}}+\|\nabla u\|_{L^2}).
	\end{aligned}
\end{equation}
By \eqref{2.1} and \eqref{2.18}, we can obtain
\begin{equation}\label{2.20}
	\begin{aligned}
		\|\nabla{\rm curl}u\|_{L^p}&\le C(\|(\rho+m)\dot{u}\|_{L^p}+\|{\rm curl}u\|_{L^p})\\
		&\le C(\|(\rho+m)\dot{u}\|_{L^p}+\|\rm{curl}u\|_{L^2}+\|\nabla{\rm curl}u\|_{L^2})\\
		&\le C(\|(\rho+m)\dot{u}\|_{L^p}+\|(\rho+m)\dot{u}\|_{L^2}+\|\nabla u\|_{L^2}),
	\end{aligned}
\end{equation}
for any $2\le p\le6$.

Employing \eqref{1.19}, \eqref{2.1}, \eqref{2.16} and \eqref{2.20}, one has
\begin{align*}
	\|F\|_{L^p}\le&C(\|F\|_{L^2}^{\frac{6-p}{2p}}\|\nabla F\|_{L^2}^{\frac{3p-6}{2p}})\\
	\le&C(\|\nabla u\|_{L^2}+\|P-P_\infty\|_{L^2})^{\frac{6-p}{2p}}\|(\rho+m)\dot{u}\|_{L^2}^{\frac{3p-6}{2p}},
\end{align*}
and
\begin{equation}\label{2.21}
	\begin{aligned}
	\|{\rm curl}u\|_{L^p}\le&C(\|{\rm curl}u\|_{L^2}^{\frac{6-p}{2p}}\|\nabla{\rm curl}u\|_{L^2}^{\frac{3p-6}{2p}})\\
	\le&C\|\nabla u\|_{L^2}^{\frac{6-p}{2p}}(\|(\rho+m)\dot{u}\|_{L^2}+\|\nabla u\|_{L^2})^{\frac{3p-6}{2p}}\\
	\le&C\|\nabla u\|_{L^2}^{\frac{6-p}{2p}}\|(\rho+m)\dot{u}\|_{L^2}^{\frac{3p-6}{2p}}+C\|\nabla u\|_{L^2}.	
	\end{aligned}
\end{equation}
Combining Lemma \ref{Lm2.6} with \eqref{2.1}, \eqref{2.16} and \eqref{2.21} gives that
\begin{equation*}
	\begin{aligned}
		\|\nabla u\|_{L^p}\le&C\|\nabla u\|_{L^2}^{\frac{6-p}{2p}}\|\nabla u\|_{L^6}^{\frac{3p-6}{2p}}\\
		\le&C\|\nabla u\|_{L^2}^{\frac{6-p}{2p}}(\|{\rm div}u\|_{L^6}+\|{\rm curl}u\|_{L^6}+\|\nabla u\|_{L^2})^{\frac{3p-6}{2p}}\\
		\le&C\|\nabla u\|_{L^2}^{\frac{6-p}{2p}}(\|F\|_{L^6}+\|P-P_\infty\|_{L^6}+\|(\rho+m)\dot{u}\|_{L^2}+\|\nabla u\|_{L^2})^{\frac{3p-6}{2p}}\\
		\le&C\|\nabla u\|_{L^2}^{\frac{6-p}{2p}}(\|(\rho+m)\dot{u}\|_{L^2}+\|P-P_\infty\|_{L^6}+\|\nabla u\|_{L^2})^{\frac{3p-6}{2p}}\\
		\le&C\|\nabla u\|_{L^2}^{\frac{6-p}{2p}}(\|(\rho+m)\dot{u}\|_{L^2}+\|P-P_\infty\|_{L^6})^{\frac{3p-6}{2p}}+C\|\nabla u\|_{L^2}.
	\end{aligned}	
\end{equation*}
\end{proof}
\section{A priori estimates~(\uppercase\expandafter{\romannumeral1}): lower order estimates}\label{S3}
Assume the exterior of a simply connected domain $\Omega\subset\mathbb{R}^3$, we can choose a positive real number $R$ such that $\bar{D}\subset B_R$, one can extend the unit outer normal $n$ to $\Omega$ as
\begin{equation}\label{3.1}
	n\in C^3(\bar{\Omega}),\ \ n\equiv0\ on\ \mathbb{R}^3\backslash B_{2R}.
\end{equation}

Then we will establish some necessary a priori bounds for smooth solutions to the problem \eqref{1.1}--\eqref{1.5} to extend the local classical solution guaranteed by Lemma \ref{Lm2.1}. Thus, let $T>0$ be a fixed time and $(\rho,m,u)$ be the smooth solution to \eqref{1.1}--\eqref{1.5} on $\Omega\times(0,T]$ with smooth initial date $(\rho_0,m_0,u_0)$ satisfying \eqref{1.9} and \eqref{1.10}. In order to estimate this solution, set $\sigma(t)\triangleq {\rm min}\left\{1,t\right\}$, and we define
\begin{equation}\label{3.2}
	A_1(T)\triangleq\mathop{\rm sup}\limits_{0\le t\le T}(\sigma\|\nabla u\|^2_{L^2})+\int_{0}^{T}\int\sigma(\rho+m)|\dot{u}|^2dxdt,
	\end{equation}
\begin{equation}\label{3.3}
	A_2(T)\triangleq\mathop{\rm sup}\limits_{0\le t\le T}\sigma^3\int(\rho+m)|\dot{u}|^2dx+\int_{0}^{T}\int\sigma^3|\nabla\dot{u}|^2dxdt,
\end{equation}
and
\begin{equation}\label{3.4}
	A_3(T)\triangleq\mathop{\rm sup}\limits_{0\le t\le T}\int|\nabla u|^2dx.
\end{equation}

Then, to get the existence of a global classical solution of \eqref{1.1}--\eqref{1.5}, we will give the useful discuss.
\begin{proposition}\label{Prop3.1}
Under the conditions of Theorem \ref{1.1}, there exists a positive constant $\varepsilon$ depending only on $\lambda$, $\mu$, $\gamma$, $\alpha$, $\Omega$, $\bar{\rho}$, $\bar{m}$ and M such that if $(\rho,m,u)$ is a smooth solution of \eqref{1.1}--\eqref{1.5} on $\Omega\times(0,T]$ satisfying
\begin{equation}\label{3.5}
	\begin{aligned}
	\mathop{\rm sup}\limits_{\Omega\times[0,T]}\rho\le2\bar{\rho},\ \mathop{\rm sup}_{\Omega\times[0,T]}m\le2\bar{m},\ A_1(T)+A_2(T)\le2C_0^{\frac{1}{2}},\ A_3(\sigma(T))\le2M,
	\end{aligned}
\end{equation}
then
\begin{equation}\label{3.6}
	\mathop{\rm sup}\limits_{\Omega\times[0,T]}\rho\le7\bar{\rho}/4,\ \mathop{\rm sup}\limits_{\Omega\times[0,T]}m\le7\bar{m}/4,\ A_1(T)+A_2(T)\le C_0^{\frac{1}{2}},\ A_3(\sigma(T))\le M,
\end{equation}
provided $C_0\le\varepsilon$.
\end{proposition}
\begin{proof}
	Proposition \ref{Prop3.1} is deduced from Lemmas \ref{Lm3.4}--\ref{Lm3.7}.
\end{proof}


 First, we start with the standard energy estimate of $(\rho,m,u)$.
\begin{lemma}\label{Lm3.2}
Suppose $(\rho,m,u)$ is a smooth solution of \eqref{1.1}--\eqref{1.5} on $\Omega\times(0,T]$. Then there is a positive constant C depending only on $\lambda$, $\mu$ and $\Omega$ such that
\begin{equation}\label{3.7}
	\mathop{\rm sup}\limits_{0\le t\le T}\int\left((\rho+m)|u|^2+G(\rho,m)\right)dx+\int_{0}^{T}\|\nabla u\|^2_{L^2}dt\le CC_0.
\end{equation}
\end{lemma}
\begin{proof}
	\par First, due to $-\Delta u=-\nabla{\rm div}u+\nabla\times{\rm curl}u$, rewrite the equation of $\eqref{1.1}_3$ as
	\begin{equation}\label{3.8}
		(\rho+m)\dot{u}-(\lambda+2\mu)\nabla{\rm div}u+\mu\nabla\times{\rm curl}u+\nabla P=0.
	\end{equation}
Multiplying \eqref{3.8} by $u$ and integrating the resultant equation over $\Omega$, we obtain that
\begin{equation}\label{3.9}
	\frac{1}{2}\left(\int(\rho+m)|u|^2dx\right)_t+(\lambda+2\mu)\|{\rm div}u\|^2_{L^2}+\mu\|{\rm curl}u\|^2_{L^2}+\int u\cdot\nabla Pdx=0.
\end{equation}
Multiplying $\eqref{1.1}_1$ by $\left(\int_{{\rho_\infty}}^{\rho}\frac{P(s,m)-P(\rho_\infty,m)}{s^2}ds+\frac{P(\rho,m)-P(\rho_\infty,m)}{\rho}\right)$ and using \eqref{1.4}, we have
\begin{equation}\label{3.10}
	\left(\int\rho\int_{{\rho_\infty}}^{\rho}\frac{P(s,m)-P(\rho_\infty,m)}{s^2}dsdx\right)_t+\int(P(\rho,m)-P(\rho_\infty,m)){\rm div}udx=0.
\end{equation}
By the same way, $\eqref{1.1}_2$ shows that
\begin{equation}\label{3.11}
	\left(\int m\int_{{m_\infty}}^{m}\frac{P(\rho,s)-P(\rho,m_\infty)}{s^2}dsdx\right)_t+\int(P(\rho,m)-P(\rho,m_\infty)){\rm div}udx=0.
\end{equation}

Combining \eqref{3.9}, \eqref{3.10} and \eqref{3.11}, we have
\begin{equation}\label{3.111}
	\left(\int\left(\frac{1}{2}(\rho+m)|u|^2+G(\rho,m)\right)dx\right)_t+(\lambda+2\mu)\|{\rm div}u\|_{L^2}^2+\mu\|{\rm curl}u\|^2_{L^2}=0.
\end{equation}
Integrating \eqref{3.111} over $(0,T]$ and using \eqref{2.7} yield \eqref{3.7}. This completes the proof.
\end{proof}
\begin{lemma}\label{Lm3.3}
Suppose $(\rho,m,u)$ is a smooth solution of \eqref{1.1}--\eqref{1.5} satisfying \eqref{3.5} on $\Omega\times(0,T]$. Then there is a positive constant C depending on $\lambda$, $\mu$, $\gamma$, $\alpha$, $\bar{\rho}$, $\bar{m}$, M and $\Omega$ such that
\begin{equation}\label{3.12}
	A_1(T)\le CC_0+C\int_{0}^{T}\int\sigma|\nabla u|^3dxdt,
\end{equation}
and
\begin{equation}\label{3.13}
	A_2(T)\le CC_0+CA_1(T)+C\int_{0}^{T}\int\sigma^3|\nabla u|^4dxdt.
	\end{equation}
\end{lemma}
\begin{proof}
	Motivated by Hoff \cite{H1995} and Cai--Li--L\"u \cite{CLL2021}. For $h\ge0$, multiplying $\eqref{1.1}_3$ by $\sigma^h\dot{u}$, and then integrating it over $\Omega$ lead to
	\begin{equation}\label{3.14}
		\begin{aligned}
		\int\sigma^h(\rho+m)|\dot{u}|^2dx
		&=(\lambda+2\mu)\int\nabla{\rm div}u\cdot\sigma^h\dot{u}dx-\mu\int\nabla\times{\rm curl}u\cdot\sigma^h\dot{u}dx-\int \nabla P\cdot\sigma^h\dot{u}dx\\
		&=:I_1+I_2+I_3.
		\end{aligned}
	\end{equation}
By using \eqref{1.21} and the fact that ${\rm div}(u\cdot\nabla u)=\nabla u:\nabla u+u\cdot\nabla{\rm div}u$, a direct calculation gives
\begin{equation}\label{3.15}
	\begin{aligned}
I_1&=(\lambda+2\mu)\int\nabla{\rm div}u\cdot\sigma^h\dot{u}dx\\
		&=-(\lambda+2\mu)\int\sigma^h{\rm div}u{\rm div}\dot{u}dx+(\lambda+2\mu)\int_{\partial\Omega}\sigma^h{\rm div}u\dot{u}\cdot nds\\
		&=-(\lambda+2\mu)\int[\sigma^h{\rm div}u{\rm div}u_t+\sigma^h{\rm div}u{\rm div}(u\cdot\nabla u)]dx+(\lambda+2\mu)\int_{\partial\Omega}\sigma^h{\rm div}uu\cdot\nabla u\cdot nds\\
		&\le-\frac{\lambda+2\mu}{2}\left(\int\sigma^h({\rm div}u)^2dx\right)_t+Ch\sigma^{h-1}\sigma'\int({\rm div}u)^2dx-(\lambda+2\mu)\int\sigma^h{\rm div}u\nabla u:\nabla udx\\
		&\quad-\frac{\lambda+2\mu}{2}\int\sigma^hu\cdot\nabla({\rm div}u)^2dx-(\lambda+2\mu)\int_{\partial\Omega}\sigma^h{\rm div}uu\cdot\nabla n\cdot uds \\
		&\le-\frac{\lambda+2\mu}{2}\left(\int\sigma^h({\rm div}u)^2dx\right)_t+\delta\|(\rho+m)^{\frac{1}{2}}\dot{u}\|^2_{L^2}+Ch\sigma^{h-1}\sigma'\|\nabla u\|^2_{L^2}\\
&\quad+C\sigma^h\left(\|\nabla u\|_{L^2}^2+\|\nabla u\|^4_{L^2}+\|\nabla u\|^3_{L^3}\right).
	\end{aligned}
\end{equation}
For the last boundary term on the right--hand side of \eqref{3.15}, it follows from \eqref{1.19}, \eqref{2.11}, \eqref{3.5} and Young's inequality that
\begin{align*}
		-(\lambda+2\mu)\int_{\partial\Omega}{\rm div}uu\cdot\nabla n \cdot uds&=-\int_{\partial\Omega}Fu\cdot\nabla n \cdot uds-\int_{\partial\Omega}(P-{P}_\infty)u\cdot\nabla n \cdot uds\\
		&\le C\left(\|F\|_{L^2(\partial\Omega)}\|u\|_{L^4(\partial\Omega)}^2+\|u\|^2_{L^2(\partial\Omega)}\right)\\
		&\le C\left(\|\nabla F\|_{L^2}\|\nabla u\|_{L^2}^2+\|\nabla u\|^2_{L^2}\right)\\
		&\le\delta\|(\rho+m)^{\frac{1}{2}}\dot{u}\|^2_{L^2}+C\left(\|\nabla u\|_{L^2}^2+\|\nabla u\|^4_{L^2}\right).
\end{align*}
Notice that ${\rm curl}(u\cdot\nabla u)=\nabla u^i\times\nabla_iu+u\cdot\nabla{\rm curl}u$, by \eqref{1.4} we have
\begin{equation}\label{3.16}
	\begin{aligned}
	I_2=&-\mu\int\sigma^h\nabla\times{\rm curl}u\cdot\dot{u}dx\\
	=&-\mu\sigma^h\int{\rm curl}u\cdot{\rm curl}u_tdx-\mu\sigma^h\int{\rm curl}u\cdot{\rm curl}(u\cdot\nabla u)dx-\mu\sigma^h\int_{\partial\Omega}{\rm curl}u\times\dot{u}\cdot nds\\
	=&-\frac{\mu}{2}\left(\sigma^h\int({\rm curl}u)^2dx\right)_t+\frac{\mu}{2}h\sigma^{h-1}\sigma'\int({\rm curl}u)^2dx-\mu\sigma^h\int{\rm curl}u\cdot(\nabla u^i\times\nabla_iu)dx\\
	&-\mu\int u\cdot\nabla\left({\frac{({\rm curl}u)^2}{2}}\right)dx\\
	\le&-\frac{\mu}{2}\left(\sigma^h\int({\rm curl}u)^2dx\right)_t+Ch\sigma^{h-1}\sigma'\|\nabla u\|^2_{L^2}+C\sigma^h\|\nabla u\|^3_{L^3}.
\end{aligned}
\end{equation}

Finally, a direct calculation leads to
\begin{equation}\label{3.17}
\begin{aligned}
	I_3=&-\int\sigma^h\dot{u}\cdot\nabla Pdx \\
	=&\int\sigma^h(P-{P}_\infty){\rm div}u_tdx+\int\sigma^h(P-{P}_\infty){\rm div}(u\cdot\nabla u)dx-\int_{\partial\Omega}\sigma^h(P-{P}_\infty)u\cdot\nabla u\cdot nds \\
	=&\left(\sigma^h\int(P-{P}_\infty){\rm div}udx\right)_t-h\sigma^{h-1}\sigma'\int(P-{P}_\infty){\rm div}udx-\sigma^h\int P_t{\rm div}udx \\
	&+\int\sigma^h(P-{P}_\infty)(\nabla u:\nabla u+u\cdot\nabla{\rm div}u)dx+\int_{\partial\Omega}\sigma^h(P-{P}_\infty)u\cdot\nabla n\cdot uds\\
	=&\left(\sigma^h\int(P-{P}_\infty){\rm div}udx\right)_t-h\sigma^{h-1}\sigma'\int(P-{P}_\infty){\rm div}udx+\int\sigma^h(P-{P}_\infty)\nabla u:\nabla udx\\
	&+\sigma^h\int((\gamma-1)\rho^\gamma+(\alpha-1)m^{\alpha}+P_\infty)({\rm div}u)^2dx+\int_{\partial\Omega}\sigma^h(P-{P}_\infty)u\cdot\nabla n\cdot uds\\
	\le&\left(\sigma^h\int(P-{P}_\infty){\rm div}udx\right)_t+Ch\sigma^{h-1}\sigma'(\|P-{P}_\infty\|^2_{L^2}+\|\nabla u\|^2_{L^2})+C\sigma^h\|\nabla u\|^2_{L^2}.
\end{aligned}
\end{equation}
where  we have used the fact that
\begin{align*}
	-\sigma^h\!\!\int P_t{\rm div}udx=&\sigma^h\int\left({\rm div}(Pu)+(\gamma-1)\rho^{\gamma}{\rm div}u+(\alpha-1)m^{\alpha}{\rm div}u\right){\rm div}udx\\
	=&\sigma^h\!\!\int{\rm div}(Pu){\rm div}udx+\sigma^h\!\!\int(\gamma-1)\rho^{\gamma}({\rm div}u)^2dx+\sigma^h\!\!\int(\alpha-1)m^{\alpha}({\rm div}u)^2dx.
\end{align*}
Combining \eqref{3.14} and \eqref{3.15}--\eqref{3.17} gives that for enough small $\delta$.
\begin{equation}\label{3.18}
	\begin{aligned}
		&\left(\frac{\lambda+2\mu}{2}\sigma^h\!\!\int({\rm div}u)^2dx+\frac{\mu}{2}\sigma^h\!\!\int({\rm curl}u)^2dx\right)_t+\sigma^h\!\!\int(\rho+m)|\dot{u}|^2dx-\left(\sigma^h\!\!\int(P-{P}_\infty){\rm div}udx\right)_t\\
		&\le\delta\sigma^h\|(\rho\!+\!m)^{\frac{1}{2}}\dot{u}\|^2_{L^2}\!+\!Ch\sigma^{h-1}\sigma'(\|\nabla u\|^2_{L^2}\!+\!\|P\!-\!P_\infty\|_{L^2}^2)\!+\!C\sigma^h\!\!\left(\|\nabla u\|_{L^2}^2\!+\!\|\nabla u\|^4_{L^2}\!+\!\|\nabla u\|^3_{L^3}\right).
	\end{aligned}
\end{equation}
Integrating \eqref{3.18} over $(0,T]$, by Lemma \ref{Lm2.6}, \eqref{3.5} and \eqref{3.7}, for $h\ge1$,  we have
\begin{align*}
	&\sigma^h\|\nabla u\|^2_{L^2}+\int_{0}^{T}\sigma^h\int(\rho+m)|\dot{u}|^2dx\\
	&\le CC_0+C\int_{0}^{T}\sigma^h\|\nabla u\|^3_{L^3}dt+C\int_{0}^{T}\sigma^h\|\nabla u\|^4_{L^2}dt,
\end{align*}
where we have used $\int_{0}^{T}h\sigma^{h-1}\sigma'\|P-P_\infty\|_{L^2}^2dt\le CC_0$.
Choosing $h=1$, and using \eqref{3.5} and \eqref{3.7}, we get \eqref{3.12}.

Now, we prove \eqref{3.13}. Applying the operator $\sigma^h\dot{u}^j[\partial/\partial t+{\rm div}(u\cdot)]$ to $\eqref{1.20}^j$, summing all the equalities with respect to $j$, and integrating over $\Omega$, we obtain
\begin{equation}\label{3.19}
	\begin{aligned}
		&\frac{1}{2}\left(\sigma^h\int(\rho+m)|\dot{u}|^2dx\right)_t-\frac{1}{2}h\sigma^{h-1}\sigma'\int(\rho+m)|\dot{u}|^2dx\\
		&=\int\sigma^h[\dot{u}\cdot\nabla F_t+\dot{u}^j{\rm div}(u\partial_jF)]dx-\mu\int\sigma^h[\dot{u}\cdot\nabla\times{\rm curl}u_t+\dot{u}^j{\rm div}(u(\nabla\times{\rm curl}u)^j)]dx\\
		&=:J_1+J_2.
	\end{aligned}
\end{equation}
A direct computation for $J_1$ shows that{\small
\begin{align}\label{3.20}
	J_1=&\int\sigma^h[\dot{u}\cdot\nabla F_t+\dot{u}^j{\rm div}(u\partial_jF)]dx\notag\\
	=&-\int\sigma^h{\rm div}\dot{u}F_tdx+\int\sigma^h\dot{u}\cdot\nabla{\rm div}(uF)dx-\int\sigma^h\dot{u}^j{\rm div}(\partial_juF)dx+\int_{\partial\Omega}\sigma^hF_tu\cdot\nabla u\cdot nds\notag\\
	=&-(2\mu+\lambda)\int\sigma^h{\rm div}\dot{u}{\rm div}u_tdx+\int\sigma^h{\rm div}\dot{u}P_tdx-\int\sigma^h{\rm div}\dot{u}{\rm div}(uF)dx+\int_{\partial\Omega}\sigma^h{\rm div}(uF)\dot{u}\cdot nds\notag\\
	&-\int\sigma^hF\dot{u}\cdot\nabla{\rm div} udx-\int\sigma^h\dot{u}\cdot\nabla u\cdot\nabla Fdx+\int_{\partial\Omega}\sigma^hF_tu\cdot\nabla u\cdot nds\notag\\
	=&-(2\mu+\lambda)\int\sigma^h({\rm div}\dot{u})^2dx+(2\mu+\lambda)\int\sigma^h{\rm div}\dot{u}\nabla u:\nabla udx+\int\sigma^h{\rm div}\dot{u}u\cdot\nabla\left({F+P-{P}_\infty}\right)dx\notag\\
	&-\int\sigma^h{\rm div}\dot{u}(u\cdot\nabla P+\gamma\rho^\gamma{\rm div}u+\alpha m^{\alpha}{\rm div}u)dx-\int\sigma^h F{\rm div}\dot{u}{\rm div}udx-\int\sigma^h{\rm div}\dot{u}u\cdot\nabla Fdx\notag\\
	&+\int_{\partial\Omega}\sigma^h{\rm div}(uF)u\cdot\nabla u\cdot nds-\frac{1}{2\mu+\lambda}\int\sigma^hF\dot{u}\cdot\nabla Fdx-\frac{1}{2\mu+\lambda}\int_{\partial\Omega}\sigma^h(P-P_{\infty})F\dot{u}\cdot ndx\notag\\
	&+\frac{1}{2\mu+\lambda}\int\sigma^h(F{\rm div}\dot{u}+\dot{u}\cdot\nabla F)(P-P_\infty)dx+\int_{\partial\Omega}\sigma^hF_tu\cdot\nabla u\cdot nds-\int\sigma^h\dot{u}\cdot\nabla u\cdot\nabla Fdx\notag\\
	=&-(2\mu+\lambda)\int[\sigma^h({\rm div}\dot{u})^2-\sigma^h{\rm div}\dot{u}\nabla u:\nabla u]dx-\int\sigma^h{\rm div}\dot{u}(\gamma\rho^\gamma{\rm div}u+\alpha m^{\alpha}{\rm div}u)dx\notag\\
	&-\int\sigma^h F{\rm div}\dot{u}{\rm div}udx+\int_{\partial\Omega}\sigma^h{\rm div}(uF)u\cdot\nabla u\cdot nds-\frac{1}{2\mu+\lambda}\int\sigma^hF\dot{u}\cdot\nabla Fdx\notag\\
	&+\frac{1}{2\mu+\lambda}\int\sigma^h\dot{u}\cdot\nabla F(P-P_\infty)dx+\frac{1}{2\mu+\lambda}\int\sigma^hF{\rm div}\dot{u}(P-P_\infty)dx\notag\\
	&-\frac{1}{2\mu+\lambda}\int_{\partial\Omega}\sigma^h(P-P_{\infty})F\dot{u}\cdot nds+\int_{\partial\Omega}\sigma^hF_tu\cdot\nabla u\cdot nds-\int\sigma^h\dot{u}\cdot\nabla u\cdot\nabla Fdx.
\end{align}}
\!\!Setting $u^{\bot}\triangleq-u\times n$, we have $u=u^{\bot}\times n$. Applying \eqref{2.11}, \eqref{2.13}, we can estimate the three boundary terms as
\begin{align}\label{3.21}
		&\int_{\partial\Omega}\sigma^hF_tu\cdot\nabla u\cdot nds\notag\\
		&=-\left(\int_{\partial\Omega}\sigma^hFu\cdot\nabla n\cdot uds\right)_t+h\sigma^{h-1}\sigma'\int_{\partial\Omega}Fu\cdot\nabla n\cdot uds+\int_{\partial\Omega}\sigma^hFu\cdot(\nabla n+(\nabla n)^{tr})\cdot u_tds\notag\\
		&=-\left(\int_{\partial\Omega}\sigma^hFu\cdot\nabla n\cdot uds\right)_t+h\sigma^{h-1}\sigma'\int_{\partial\Omega}Fu\cdot\nabla n\cdot uds+\int_{\partial\Omega}\sigma^hF{u}\cdot(\nabla n+(\nabla n)^{tr})\cdot \dot{u}ds\notag\\
		&\quad-\int_{\partial\Omega}\sigma^hFu\cdot(\nabla n+(\nabla n)^{tr})^i(u^{\bot}\times n\cdot\nabla u^i)ds\notag\\
		&\le-\!\left(\int_{\partial\Omega}\!\sigma^hFu\cdot\nabla n\cdot uds\right)_t\!+\!C(h\sigma^{h-1}\sigma'\|F\|_{L^2(\partial\Omega)}\|u\|_{L^4(\partial\Omega)}^2\!+\!\sigma^h\|F\|_{L^4(\partial\Omega)}\|u\|_{L^4(\partial\Omega)}\|\dot{u}\|_{L^2(\partial\Omega)})\notag\\
		&\quad+\int(|u|^2|\nabla u||\nabla F|+|\nabla u|^2|u||F|+|u|^2|\nabla u||F|)dx\notag\\
		&\le-\left(\int_{\partial\Omega}\sigma^hFu\cdot\nabla n\cdot uds\right)_t+C(h\sigma^{h-1}\sigma'\|\nabla F\|_{L^2}\|\nabla u\|_{L^2}^2+\sigma^h\|\nabla F\|_{L^2}\|\nabla u\|_{L^2}\|\nabla\dot{u}\|_{L^2})\notag\\
		&\quad+C\sigma^h(\|\nabla F\|_{L^6}\|\nabla u\|_{L^2}\|u\|_{L^6}^2+\|F\|_{L^6}\|\nabla u\|_{L^3}^2\|u\|_{L^6}^2+\|F\|_{L^6}\|\nabla u\|_{L^2}\|u\|_{L^6}^2)\notag\\
		&\le-\left(\int_{\partial\Omega}\sigma^hFu\cdot\nabla n\cdot uds\right)_t+Ch\sigma^{h-1}\sigma'\|(\rho+m)^{\frac{1}{2}}\dot{u}\|_{L^2}\|\nabla u\|_{L^2}^2+\delta\sigma^h\|\nabla\dot{u}\|_{L^2}^2\notag\\
		&\quad+C\sigma^h(\|(\rho+m)^{\frac{1}{2}}\dot{u}\|_{L^2}^2(\|\nabla u\|_{L^2}^4+1)+\|\nabla u\|_{L^2}^2+\|\nabla u\|_{L^2}^6+\|\nabla u\|_{L^4}^4),
\end{align}
where we have used
\begin{align*}
	&-\int_{\partial\Omega}\sigma^hFu\cdot(\nabla n+(\nabla n)^{tr})^i(u^{\bot}\times n\cdot\nabla u^i)ds\\
	&=\sigma^h\int{\rm div}(Fu\cdot(\nabla n+(\nabla n)^{tr})^iu^{\bot}\times\nabla u^i)dx\\
	&=\sigma^h\int u^{\bot}\times\nabla u^i\cdot\nabla(Fu\cdot(\nabla n+(\nabla n)^{tr})^i)dx+\sigma^h\int(\nabla\times u^{\bot}\cdot\nabla  u^i)(Fu\cdot(\nabla n+(\nabla n)^{tr})^i)dx.
\end{align*}
And
\begin{equation}\label{3.22}
	\begin{split}
	&\int_{\partial\Omega}\sigma^h{\rm div}(uF)u\cdot\nabla u\cdot nds-\frac{1}{2\mu+\lambda}\int_{\partial\Omega}\sigma^h(P-P_{\infty})F\dot{u}\cdot nds\\
	&=\int_{\partial\Omega}\sigma^h(u\cdot\nabla F)(u\cdot\nabla u\cdot n)ds+\frac{1}{2\mu+\lambda}\int_{\partial\Omega}\sigma^hF^2u\cdot\nabla u\cdot nds\\
	&=-\int_{\partial\Omega}\sigma^h(u^{\bot}\times n\cdot\nabla F)(u\cdot\nabla n\cdot u)ds-\frac{1}{2\mu+\lambda}\int_{\partial\Omega}\sigma^hF^2u\cdot\nabla n\cdot uds\\
	&=\int\sigma^h{\rm div}((u\cdot\nabla n\cdot u)(u^{\bot}\times\nabla F))dx-\frac{1}{2\mu+\lambda}\int_{\partial\Omega}\sigma^hF^2u\cdot\nabla n\cdot uds\\
	&=\int\left[\sigma^h\nabla(u\cdot\nabla n\cdot u)\cdot(u^{\bot}\times\nabla F)+(u\cdot\nabla n\cdot u)\nabla\times u^{\bot}\cdot\nabla F\right]dx+\frac{1}{2\mu+\lambda}\int_{\partial\Omega}\sigma^hF^2u\cdot\nabla n\cdot uds\\
	&\le\int\sigma^h(|\nabla u||u|^2|\nabla F|+|u|^3|\nabla F|)dx+C\sigma^h\|F\|_{L^4(\partial\Omega)}^2\|u\|_{L^4(\partial\Omega)}^2\\
	&\le C\sigma^h(\|\nabla u\|_{L^2}\|u\|^2_{L^6}\|\nabla F\|_{L^6}+\|u\|^3_{L^6}\|\nabla F\|_{L^2}+\|\nabla F\|^2_{L^2}\|\nabla u\|^2_{L^2})\\
	&\le\sigma^h(\|(\rho+m)^{\frac{1}{2}}\dot{u}\|^2_{L^2}(\|\nabla u\|_{L^2}^2+\|\nabla u\|_{L^2}^4)+\|\nabla u\|_{L^2}^2+\|\nabla u\|_{L^2}^6)+\delta\sigma^h\|\nabla\dot{u}\|_{L^2}^2.
	\end{split}
\end{equation}
It follows from \eqref{2.11}, \eqref{2.13} and \eqref{3.20}--\eqref{3.22} that
\begin{equation}\label{3.23}
	\begin{aligned}
		J_1&\le-\left(\int_{\partial\Omega}\sigma^hFu\cdot\nabla n\cdot uds\right)_t-(2\mu+\lambda)\int\sigma^h({\rm div}\dot{u})^2dx+Ch\sigma^{h-1}\sigma'\|(\rho+m)^{\frac{1}{2}}\dot{u}\|_{L^2}\|\nabla u\|_{L^2}^2\\
		&\quad+\delta\sigma^h\|\nabla\dot{u}\|_{L^2}^2+C\sigma^h(\|(\rho+m)^{\frac{1}{2}}\dot{u}\|_{L^2}^2(\|\nabla u\|_{L^2}^2+\|\nabla u\|_{L^2}^4)+\|\nabla u\|_{L^2}^2+\|\nabla u\|_{L^2}^6+\|\nabla u\|_{L^4}^4)\\
		&\quad+C\sigma^h\left[\|\nabla\dot{u}\|_{L^2}(\|\nabla u\|^2_{L^4}+\|\nabla u\|_{L^2})+(\|F\|_{L^2}\|\dot{u}\|_{L^6}\|\nabla F\|_{L^3}+\|\nabla F\|_{L^2}\|\dot{u}\|_{L^6}\|P\!-\!P_\infty\|_{L^3}\right]\\
&\quad+\|F\|_{L^6}\|\nabla\dot{u}\|_{L^2}\|P\!-\!P_{\infty}\|_{L^3}+\|\dot{u}\|_{L^6}\|\nabla F\|_{L^3}\|\nabla u\|_{L^2})\\
		&\le-\left(\int_{\partial\Omega}\sigma^hFu\cdot\nabla n\cdot uds\right)_t-(2\mu+\lambda)\int\sigma^h({\rm div}\dot{u})^2dx+Ch\sigma^{h-1}\sigma'\|(\rho+m)^{\frac{1}{2}}\dot{u}\|_{L^2}\|\nabla u\|_{L^2}^2\\
		&\quad+\delta\sigma^h\|\nabla\dot{u}\|_{L^2}^2+C\sigma^h(\|(\rho+m)^{\frac{1}{2}}\dot{u}\|_{L^2}^2(1+\|\nabla u\|_{L^2}^4)+\|\nabla u\|_{L^2}^2+\|\nabla u\|_{L^2}^6+\|\nabla u\|_{L^4}^4).
	\end{aligned}
\end{equation}
For the term $J_2$, a direct computation yields{\small
\begin{align}\label{3.24}
		J_2&=-\mu\int\sigma^h\dot{u}\cdot\nabla\times{\rm curl}u_tdx-\mu\int\sigma^h\dot{u}^j{\rm div}((\nabla\times{\rm curl}u)^ju)dx\notag\\
		&=-\mu\int\sigma^h{\rm curl}\dot{u}{\rm curl}u_tdx+\mu\int_{\partial\Omega}\sigma^h{\rm curl}u_t\times\dot{u}\cdot nds-\mu\int\sigma^h\dot{u}\cdot(\nabla\times{\rm curl}u){\rm div}udx\notag\\
		&\quad-\mu\int\sigma^h{u}^i\dot{u}\cdot\nabla\times(\nabla_i{\rm curl}u)dx\notag\\
		&=-\mu\int\sigma^h({\rm curl}\dot{u})^2dx+\mu\int\sigma^h{\rm curl}\dot{u}\cdot(\nabla u^i\times\nabla^iu)dx+\mu\int\sigma^h{\rm curl}\dot{u}\cdot(u\cdot\nabla{\rm curl}u)dx\notag\\
		&\quad-\mu\int\sigma^h{\rm div}u{\rm curl}\dot{u}\cdot{\rm curl}udx-\mu\int\sigma^h\nabla{\rm div}u\times\dot{u}\cdot{\rm curl}udx-\mu\int_{\partial\Omega}\sigma^h{\rm curl}u\times({\rm div}u\dot{u})\cdot nds\notag\\
		&\quad-\mu\int\sigma^hu\cdot\nabla{\rm curl}u\cdot{\rm curl}\dot{u}dx-\mu\int\sigma^h\nabla{u}^i\times\dot{u}\cdot(\nabla_i{\rm curl}u)dx-\mu\int_{\partial\Omega}\sigma^h\nabla_i{\rm curl}u\times(u^i\dot{u})\cdot nds\notag\\
		&=\quad-\mu\int\sigma^h({\rm curl}\dot{u})^2dx+\mu\int\sigma^h{\rm curl}\dot{u}\cdot(\nabla u^i\times\nabla^iu)dx-\mu\int\sigma^h{\rm div}u{\rm curl}\dot{u}\cdot{\rm curl}udx\notag\\
		&\quad-\mu\int\sigma^h\nabla{\rm div}u\times\dot{u}\cdot{\rm curl}udx-\mu\int\sigma^h{\rm curl}u\times\dot{u}\cdot\nabla{\rm div}udx-\mu\int\sigma^h{\rm div}u{\rm div}({\rm curl}u\times\dot{u})dx\notag\\
		&\quad-\mu\int\sigma^h\nabla{u}^i\times\dot{u}\cdot(\nabla_i{\rm curl}u)dx-\mu\int\sigma^h\nabla u^i\cdot(\nabla_i{\rm curl}u\times\dot{u})dx-\mu\int \sigma^hu^i{\rm div}(\nabla_i{\rm curl}u\times\dot{u})dx\notag\\
		&=-\mu\int\sigma^h({\rm curl}\dot{u})^2dx+\mu\int\sigma^h{\rm curl}\dot{u}\cdot(\nabla u^i\times\nabla^iu)dx-\mu\int\sigma^h{\rm div}u{\rm curl}\dot{u}\cdot{\rm curl}udx\notag\\
		&\quad-\mu\int\sigma^h{\rm div}u{\rm div}({\rm curl}u\times\dot{u})dx-\mu\int \sigma^hu^i\nabla^i{\rm div}({\rm curl}u\times\dot{u})dx-\mu\int\sigma^h\nabla u^i\cdot({\rm curl}u\times\nabla_i\dot{u})dx\notag\\
		&=-\mu\int\sigma^h({\rm curl}\dot{u})^2dx+\mu\int\sigma^h{\rm curl}\dot{u}\cdot(\nabla u^i\times\nabla^iu)dx-\mu\int\sigma^h{\rm div}u{\rm curl}\dot{u}\cdot{\rm curl}udx\notag\\
		&\quad-\mu\int\sigma^h\nabla u^i\cdot({\rm curl}u\times\nabla_i\dot{u})dx\notag\\
		&\le-\mu\int\sigma^h({\rm curl}\dot{u})^2dx+\delta\sigma^h\|\nabla\dot{u}\|_{L^2}^2+C\sigma^h\|\nabla u\|_{L^4}^4.
\end{align}}
\!\!Combining $u=u^{\bot}\times n$ and \eqref{1.4} gives
\begin{equation*}
	(\dot{u}-(u\cdot\nabla n)\times u^{\bot})\cdot n=0\ \text{ on }\ \partial\Omega,
\end{equation*}
which together with \eqref{2.7} and \eqref{3.1} implies
\begin{equation}\label{3.25}
	\begin{aligned}
	\|\nabla\dot{u}\|_{L^2}^2&\le\|\dot{u}-(u\cdot\nabla n)\times u^{\bot}\|_{L^2}^2+\|(u\cdot\nabla n)\times u^{\bot}\|_{L^2}^2\\
	&\le C(\|{\rm div}\dot{u}\|_{L^2}^2+\|{\rm curl}\dot{u}\|_{L^2}^2+\|(u\cdot\nabla n)\times u^{\bot}\|_{L^2}^2+\|\nabla((u\cdot\nabla n)\times u^{\bot})\|_{L^2}^2)\\
	&\le C(\|{\rm div}\dot{u}\|_{L^2}^2+\|{\rm curl}\dot{u}\|_{L^2}^2+\|u\|_{L^4(B_{2R})}^4+\|\nabla u\|_{L^4(B_{2R})}^2\|u\|_{L^4(B_{2R})}^2)\\
	&\le C(R)(\|{\rm div}\dot{u}\|_{L^2}^2+\|{\rm curl}\dot{u}\|_{L^2}^2+\|\nabla u\|_{L^4(B_{2R})}^4+\|\nabla u\|_{L^2(B_{2R})}^4.
	\end{aligned}
\end{equation}
Putting \eqref{3.23}, \eqref{3.24} and \eqref{3.25} into \eqref{3.19}, for a enough small $\delta$, we obtain
\begin{equation}\label{3.26}
	\begin{aligned}
		&\left(\int\sigma^h(\rho+m)|\dot{u}|^2dx\right)_t+\sigma^h\|\nabla\dot{u}\|^2_{L^2}+\left(\int_{\partial\Omega}\sigma^hFu\cdot\nabla n\cdot udx\right)_t\\
		&\le C\sigma^{h-1}\sigma'\|(\rho+m)^{\frac{1}{2}}\dot{u}\|_{L^2}\|\nabla u\|^2_{L^2}+C\sigma^h(\|(\rho+m)^{\frac{1}{2}}\dot{u}\|_{L^2}^2(1+\|\nabla u\|_{L^2}^4)\\
&\quad+\|\nabla u\|_{L^2}^2+\|\nabla u\|_{L^2}^6+\|\nabla u\|_{L^4}^4).
	\end{aligned}
\end{equation}
Integrating \eqref{3.26} over $(0,T]$ and  using \eqref{2.11} and \eqref{3.5}, when $h\ge3$, we have
\begin{align}\label{3.42}
		&\sigma^h\int(\rho+m)|\dot{u}|^2dx+\int_{0}^{T}\sigma^h\|\nabla\dot{u}\|^2_{L^2}dt\notag\\
		&\le-\int_{\partial\Omega}\sigma^hFu\cdot\nabla n\cdot uds+CA_1(T)+CC_0+C\int_{0}^{T}\sigma^h\|\nabla u\|^4_{L^4}dt\\
		&\le\delta\sigma^h\|(\rho+m)^{\frac{1}{2}}\dot{u}\|^2_{L^2}+CA_1(T)+CC_0+C\int_{0}^{T}\sigma^h\|\nabla u\|^4_{L^4}dt,\notag
	\end{align}
where in the last inequality we have used
\begin{align*}
	-\int_{\partial\Omega}\sigma^hFu\cdot\nabla n\cdot uds\le&C\sigma^h\|F\|_{L^2(\partial\Omega)}\|u\|_{L^4(\partial\Omega)}^2
	\le C\sigma^h\|\nabla F\|_{L^2}\|\nabla u\|_{L^2}^2\\
	\le&\delta\sigma^h\|(\rho+m)^{\frac{1}{2}}\dot{u}\|^2_{L^2}+CC_0.
\end{align*}
Then taking $h=3$ and choosing enough small $\delta$, we obtain \eqref{3.13}.
\end{proof}
\begin{lemma}\label{Lm3.4}
	Suppose that $(\rho,m,u)$ is a smooth solution of \eqref{1.1}--\eqref{1.5} on $\Omega\times(0,T]$ satisfying \eqref{3.5}. Then there exists a positive constant C depending only on $\lambda$, $\mu$, $\bar{\rho}$, $\bar{m}$, M and $\Omega$ such that
	\begin{equation}\label{3.28}
		A_3(\sigma(T))+\int_{0}^{\sigma(T)}\|(\rho+m)^{\frac{1}{2}}\dot{u}\|_{L^2}^2dt\le M,
	\end{equation}
	provided $C_0\le\varepsilon_1$.
\end{lemma}
\begin{proof}
	Taking $h=0$ in \eqref{3.18}, integrating over $(0,\sigma(T)]$, and using Lemma \ref{Lm2.6}, \eqref{2.15}, \eqref{3.5} and \eqref{3.7}, we can obtain
	\begin{equation}\label{3.29}
		\begin{aligned}
			&\|\nabla u\|_{L^2}^2+\int_{0}^{\sigma(T)}\|(\rho+m)^{\frac{1}{2}}\dot{u}\|_{L^2}^2dt\\
			&\le\frac{M}{2}+C\int_{0}^{\sigma(T)}(\|\nabla u\|_{L^2}^2+\|\nabla u\|_{L^2}^4+\|\nabla u\|_{L^3}^3)dt+\delta\|\nabla u\|_{L^2}^2+C\|P\!-\!{P}_\infty\|_{L^2}^2+CC_0\\
			&\le\frac{M}{2}+CC_0(1+M)+\delta\|\nabla u\|_{L^2}^2+C\!\!\int_{0}^{\sigma(T)}\!\!(\|\nabla{u}\|_{L^2}^\frac{3}{2}(\|(\rho+m)\dot{u}\|_{L^2}+\|P\!-\!{P}_\infty\|_{L^6})^{\frac{3}{2}}+\|\nabla u\|_{L^2}^{3})dt\\
			&\le\frac{M}{2}+CC_0(1+M)+\delta\|\nabla u\|_{L^2}^2+\delta\int_{0}^{\sigma(T)}\!\!\|(\rho+m)^{\frac{1}{2}}\dot{u}\|_{L^2}^2dt+C\!\!\int_{0}^{\sigma(T)}\!\!(\|\nabla u\|_{L^2}^6+\|P-{P}_\infty\|_{L^2}^2)dt\\
			&\le\frac{M}{2}+CC_0(1+M^2)+\delta\|\nabla u\|_{L^2}^2+\delta\int_{0}^{\sigma(T)}\|(\rho+m)^{\frac{1}{2}}\dot{u}\|_{L^2}^2dt.
		\end{aligned}	
	\end{equation}
	If we choose $\delta$ small enough, \eqref{3.29} gives
	\begin{equation*}
		A_3(\sigma(T))+\int_{0}^{\sigma(T)}||(\rho+m)^{\frac{1}{2}}\dot{u}||_{L^2}^2dt\le\frac{M}{2}+CC_0+CM^2C_0\le M,
	\end{equation*}
	provided $C_0\le\varepsilon_1\triangleq\left\{1,\frac{M}{4C},\frac{1}{4MC}\right\}$.
\end{proof}
\begin{lemma}\label{Lm3.3.1}
	Suppose that $(\rho,m,u)$ is a smooth solution of \eqref{1.1}--\eqref{1.5} on $\Omega\times(0,T]$ satisfying \eqref{3.5}. Then there exists a positive constant C depending only on $\lambda$, $\mu$, $\bar{\rho}$, $\bar{m}$, M and $\Omega$ such that
	\begin{equation}\label{3.3.1}
		\mathop{\rm sup}_{0\le t\le\sigma(T)}t\|(\rho+m)^{\frac{1}{2}}\dot{u}\|_{L^2}^2+\int_{0}^{\sigma(T)}\!\!t\|\nabla\dot{u}\|_{L^2}^2dt\le C.
	\end{equation}
\end{lemma}
\begin{proof}
	Taking $h=1$ in \eqref{3.26}, and integrating over $(0,\sigma(T)]$,  by \eqref{3.5}, \eqref{3.28} and \eqref{2.15}, we have
\begin{align*}
	&\mathop{\rm sup}_{0\le t\le\sigma(T)}t\|(\rho+m)^{\frac{1}{2}}\dot{u}\|_{L^2}^2+\int_{0}^{\sigma(T)}\!\!t\|\nabla\dot{u}\|_{L^2}^2dt\\
	&\le C\int_{0}^{\sigma(T)}\!\!t(\|(\rho+m)^{\frac{1}{2}}\dot{u}\|_{L^2}^2(1+\|\nabla u\|_{L^2}^4)+\|\nabla u\|_{L^2}^2+\|\nabla u\|_{L^6}^2+\|\nabla u\|_{L^4}^4)dt\\
	&\quad+C\int_{0}^{\sigma(T)}\|(\rho+m)^{\frac{1}{2}}\dot{u}\|_{L^2}\|\nabla u\|_{L^2}^2dt+Ct\|F\|_{L^2(\partial\Omega)}\|u\|_{L^4(\partial\Omega)}^2\\
	&\le C+C\int_{0}^{\sigma(T)}\!\![t(\|(\rho+m)\dot{u}\|_{L^2}+\|P-P_{\infty}\|_{L^6})^3\|\nabla u\|_{L^2}+t\|\nabla u\|_{L^2}^4]dt+Ct\|\nabla F\|_{L^2}\|\nabla u\|_{L^2}\\
	&\le C+C\mathop{\rm sup}_{0\le t\le\sigma(T)}(t\|(\rho+m)^{\frac{1}{2}}\dot{u}\|_{L^2}^2)^{\frac{1}{2}}\mathop{\rm sup}_{0\le t\le\sigma(T)}(\|\nabla u\|_{L^2}^2)^{\frac{1}{2}}\int_{0}^{\sigma(T)}\!\!\|(\rho+m)^{\frac{1}{2}}\dot{u}\|_{L^2}^2dt\\
&\quad+Ct\|(\rho+m)\dot{u}\|_{L^2}\|\nabla u\|_{L^2}^2\\
	&\le C+\delta\mathop{\rm sup}_{0\le t\le\sigma(T)}t\|(\rho+m)^{\frac{1}{2}}\dot{u}\|_{L^2}^2,
\end{align*}
which  gives \eqref{3.3.1} when we choose enough small $\delta$.
\end{proof}
\begin{lemma}\label{Lm3.6}
	If $(\rho,m,u)$ is a smooth solution of \eqref{1.1}--\eqref{1.5} on $\Omega\times(0, T]$ satisfying the assumption \eqref{3.5}, then it holds that	
	\begin{equation}\label{3.30}
		\int_{0}^{T}\sigma^3\|P-{P}_\infty\|_{L^4}^4dt\le CC_0.
	\end{equation}
\end{lemma}
\begin{proof}
A direct computation shows that
\begin{equation}\label{3.31}
\begin{aligned}
			&-\left(\int(P-{P}_\infty)^3dx\right)_t\\&=-3\int(P-{P}_\infty)^2P_tdx\\
			&=3\int(P-{P}_\infty)^2(\gamma\rho^\gamma+\alpha m^\alpha){\rm div}udx+3\int(P-{P}_\infty)^2u\cdot\nabla Pdx\\
			&=3\int(P-{P}_\infty)^2(\gamma\rho^\gamma+\alpha m^\alpha){\rm div}udx-\int{\rm div}u(P-P_\infty)^3dx\\
			&=3\int(P-{P}_\infty)^2(\gamma\rho^\gamma+\alpha m^\alpha){\rm div}udx-\int\frac{F}{2\mu+\lambda}(P-{P}_\infty)^3dx-\frac{1}{2\mu+\lambda}\int(P-{P}_\infty)^4dx,
\end{aligned}
\end{equation}
which indicates that
\begin{equation}\label{3.32}
\begin{aligned}
			\frac{1}{2\mu+\lambda}\sigma^3\int(P-{P}_\infty)^4dx=&\left(\sigma^3\int(P-{P}_\infty)^3dx\right)_t-3\sigma^2\sigma'\int(P-{P}_\infty)^3dx\\
			&+3\sigma^3\!\!\int({P\!-\!P_\infty})^2(\gamma\rho^\gamma+\alpha m^\alpha){\rm div}udx-\frac{\sigma^3}{2\mu+\lambda}\int F(P\!-\!{P}_\infty)^3dx.
\end{aligned}
\end{equation}
Combining \eqref{3.32}, \eqref{2.13}, \eqref{3.5} and \eqref{3.7} implies that
\begin{equation*}
\begin{aligned}
			&\int_{0}^{T}\sigma^3\|P-{P}_\infty\|_{L^4}^4dt\\
			&\le\sigma^3\|P\!-\!{P}_\infty\|_{L^3}^3+C\int_{0}^{T}\!\!\sigma'\|P\!-\!{P}_\infty\|_{L^2}^2dt+\delta\!\!\int_{0}^{T}\!\!\sigma^3\|P\!-\!{P}_\infty\|_{L^4}^4dt+C\!\!\int_{0}^{T}(\|\nabla u\|_{L^2}^2+\sigma^3\|F\|_{L^4}^4)dt\\
			&\le\delta\int_{0}^{T}\sigma^3\|P-{P}_\infty\|_{L^4}^4dt+CC_0+C\int_{0}^{T}\sigma^3(\|\nabla u\|_{L^2}+\|P-P_\infty\|_{L^2})\|(\rho+m)^{\frac{1}{2}}\dot{u}\|_{L^2}^3dt\\
			&\le\delta\int_{0}^{T}\!\!\sigma^3\|P\!-\!{P}_\infty\|_{L^4}^4dt+CC_0+C\int_{0}^{\sigma(T)}\!\!\left(\sigma^3\|(\rho+m)^{\frac{1}{2}}\dot{u}\|_{L^2}^2\right)^{\frac{1}{2}}\left(\sigma\|\nabla u\|_{L^2}^2\right)^{\frac{1}{2}}\sigma\|(\rho+m)^{\frac{1}{2}}\dot{u}\|_{L^2}^2dt\\
			&\quad+CC_0^{\frac{1}{2}}\int_{0}^{\sigma(T)}\!\!\left(\sigma^3\|(\rho+m)^{\frac{1}{2}}\dot{u}\|_{L^2}^2\right)^{\frac{1}{2}}\sigma\|(\rho+m)^{\frac{1}{2}}\dot{u}\|_{L^2}^2dt\\
			&\le\delta\int_{0}^{T}\sigma^3\|P-{P}_0\|_{L^4}^4dt+CC_0,
\end{aligned}
\end{equation*}
which yields \eqref{3.30} when we choose enough small $\delta$.
\end{proof}
\begin{lemma}\label{Lm3.7}
Suppose $(\rho,m,u)$ is a smooth solution of \eqref{1.1}--\eqref{1.5} on $\Omega\times(0, T]$ satisfying \eqref{3.5}. Then there exists a positive constant C depending only on $\lambda$, $\mu$, $\gamma$, $\alpha$, M, $\Omega$, $\bar{\rho}$ and $\bar{m}$ such that
\begin{equation}\label{3.33}
	A_1(T)+A_2(T)\le C_0^{\frac{1}{2}},
\end{equation}
provided $C_0\le\varepsilon_2$.
\end{lemma}
\begin{proof}
By \eqref{2.15}, \eqref{3.5}, \eqref{3.7} and \eqref{3.30}, it holds that
\begin{equation}\label{3.34}
	\begin{aligned}
		&\int_{0}^{T}\sigma^3\|\nabla u\|^4_{L^4}dt\\
		&\le C\int_{0}^{T}\sigma^3\|\nabla u\|_{L^2}(\|(\rho+m)^{\frac{1}{2}}\dot{u}\|_{L^2}+\|P-{P}_\infty\|_{L^6})^3dt+C\int_{0}^{T}\sigma^3\|\nabla u\|^4_{L^2}dt\\
		&\le\int_{0}^{T}\left(\sigma^3\|(\rho+m)^{\frac{1}{2}}\dot{u}\|_{L^2}^2\right)^{\frac{1}{2}}\left(\sigma\|\nabla u\|_{L^2}^2\right)^{\frac{1}{2}}\sigma\|(\rho+m)^{\frac{1}{2}}\dot{u}\|_{L^2}^2dt+C\int_{0}^{T}\sigma^3\|P-{P}_\infty\|_{L^6}^6dt\\
		&\quad+C(1+C_0^{\frac{1}{2}})\int_{0}^{T}\|\nabla u|^2_{L^2}dt\\
		&\le CC_0,
	\end{aligned}
\end{equation}
which together with \eqref{3.12} and \eqref{3.13} gives
\begin{equation}\label{3.35}
	A_1(T)+A_2(T)\le CC_0+C\int_{0}^{T}\sigma\|\nabla u\|^3_{L^3}dt.
\end{equation}	
It follows from \eqref{2.15}, \eqref{3.5}, \eqref{3.7} and \eqref{3.28} that
\begin{equation}\label{3.36}
	\begin{aligned}
	&\int_{0}^{\sigma(T)}\sigma\|\nabla u\|^3_{L^3}dt\\
	&\le C\int_{0}^{\sigma(T)}\sigma\|\nabla u\|^{\frac{3}{2}}_{L^2}\left(\|(\rho+m)^{\frac{1}{2}}\dot{u}\|_{L^2}+\|P-{P}_\infty\|_{L^6}\right)^{\frac{3}{2}}dt+C\int_{0}^{\sigma(T)}\sigma\|\nabla u\|^3_{L^2}dt\\
	&\le C\int_{0}^{\sigma(T)}\left(\|P-{P}_\infty\|^6_{L^6}+\|\nabla u\|^2_{L^2}+\sigma^{\frac{4}{3}}\|(\rho+m)^{\frac{1}{2}}\dot{u}\|^2_{L^2}\|\nabla u\|_{L^2}^{\frac{2}{3}}+\|\nabla u\|^4_{L^2}\right)dt+CC_0\\
	&\le CC_0^{\frac{2}{3}}.
	\end{aligned}
\end{equation}
On the other hand, by using \eqref{3.7}, \eqref{3.34} and Young's inequality, we can get
\begin{equation}\label{3.37}
	\begin{aligned}
		\int_{\sigma(T)}^{T}\sigma\|\nabla u\|^3_{L^3}dt&\le C\int_{\sigma(T)}^{T}\sigma\|\nabla u\|_{L^2}\|\nabla u\|_{L^4}^2dt\\
		&\le C\int_{\sigma(T)}^{T}\sigma\|\nabla u\|^2_{L^2}dt+C\int_{\sigma(T)}^{T}\sigma^3\|\nabla u\|^4_{L^4}dt\\
		&\le CC_0.
	\end{aligned}
\end{equation}
By \eqref{3.34}--\eqref{3.37}, we can obtain
\begin{equation}\label{3.38}
	A_1(T)+A_2(T)\le C\left(\bar{\rho},\bar{m},M\right)C_0^{\frac{2}{3}}\le C_0^{\frac{1}{2}},
\end{equation}
which gives \eqref{3.33} provided $C_0\le\varepsilon_2\triangleq\left\{\varepsilon_1,\left(\frac{1}{C\left(\bar{\rho},\bar{m},M\right)}\right)^6\right\}.$
\end{proof}

In order to get all the higher order estimates and to extend the classical solution globally, we must derive a uniform (in time) upper bound of the density.
\begin{lemma}\label{Lm3.8}
If $(\rho,m,u)$ is a smooth solution of \eqref{1.1}--\eqref{1.5} on $\Omega\times(0, T]$ satisfying \eqref{3.5}, then there exists a positive constant $\varepsilon$ depending only on $\lambda$, $\mu$, $\gamma$, $\alpha$, $\rho_\infty$, $m_\infty$, $\Omega$, $M$, $\bar{\rho}$ and $\bar{m}$ such that
\begin{equation}\label{3.39}
	\mathop{{\rm sup}}_{0\le t\le T}\|(\rho+m)(t)\|_{L^\infty}\le\frac{7}{4}\left(\bar{\rho}+\bar{m}\right),
\end{equation}
provided $C_0\le\varepsilon$.
\end{lemma}
\begin{proof}
 First, the equation $\eqref{1.1}_1$ and $\eqref{1.1}_2$ can be rewritten as
	\begin{equation}\label{3.40}
		D_t(\rho+m)=g(\rho+m)+b'(t)
	\end{equation}
where $D_t(\rho+m)\triangleq(\rho+m)_t+u\cdot\nabla(\rho+m),\, g(\rho+m)\triangleq-\frac{\rho+m}{2\mu+\lambda}(P-P_\infty)$, and $b(t)\triangleq-\frac{1}{2\mu+\lambda}\int_{0}^{t}(\rho+m)Fd\tau$.
On the one hand, for all $0\le t_1\le t_2\le\sigma(T)$, one deduces from \eqref{2.2}, \eqref{2.11}, \eqref{2.13}, \eqref{3.5}, \eqref{3.28} and \eqref{3.3.1} that
\begin{equation*}{\small
	\begin{aligned}
		&|b(t_2)-b(t_1)|\le C\int_{t_1}^{t_2}|(\rho+m)F|dt\le C\int_{0}^{\sigma(T)}\|F\|_{L^\infty}dt\\
		&\le C\int_{0}^{\sigma(T)}\|F\|_{L^6}^{\frac{1}{2}}\|\nabla F\|_{L^6}^{\frac{1}{2}}dt\le C\int_{0}^{\sigma(T)}\|(\rho+m)^{\frac{1}{2}}\dot{u}\|_{L^2}^{\frac{1}{2}}\|\nabla\dot{u}\|_{L^2}^{\frac{1}{2}}dt\\ &\le C\int_{0}^{\sigma(T)}(t\|(\rho+m)^{\frac{1}{2}}\dot{u}\|_{L^2}^2)^{\frac{1}{8}}\|(\rho+m)^{\frac{1}{2}}\dot{u}\|_{L^2}^{\frac{1}{4}}(t\|\nabla\dot{u}\|_{L^2}^2)^{\frac{1}{4}}t^{-\frac{3}{8}}dt\\		&\le C\left(\int_{0}^{\sigma(T)}\!\!t\|(\rho+m)^{\frac{1}{2}}\dot{u}\|_{L^2}^2dt\right)^{\frac{1}{8}}\left(\int_{0}^{\sigma(T)}\!\!\|(\rho+m)^{\frac{1}{2}}\dot{u}\|_{L^2}^2dt\right)^{\frac{1}{8}}
\left(\int_{0}^{\sigma(T)}\!\!t\|\nabla\dot{u}\|_{L^2}^2dt\right)^{\frac{1}{4}}\left(\int_{0}^{\sigma(T)}\!\!t^{-\frac{3}{4}}dt\right)^{\frac{1}{2}}\\
&\le C(\bar{\rho}+\bar{m},M)C_0^{\frac{1}{16}},
	\end{aligned}}
\end{equation*}
From Lemma \ref{Lm2.3}, we choose $N_1=0$, $N_0=C(\bar{\rho}+\bar{m},M)C_0^{\frac{1}{16}}$, and $\hat{\zeta}=\bar{\rho}+\bar{m}$ and then we use \eqref{3.40} to get
\begin{equation}\label{3.43}
	\mathop{{\rm sup}}_{0\le t\le\sigma(T)}\|\rho+m\|_{L^\infty}\le\bar{\rho}+\bar{m}+C(\bar{\rho}+\bar{m},M)C_0^{\frac{1}{16}}\le\frac{3}{2}\left(\bar{\rho}+\bar{m}\right),
\end{equation}
provided
\begin{align*}
	C_0\le\varepsilon_3\triangleq {\rm min}\left\{\varepsilon_2,\left(\frac{\bar{\rho}+\bar{m}}{2C(\bar{\rho}+\bar{m},M)}\right)^{16}\right\}.
\end{align*}

On the other hand,
for $\sigma(T)\le t_1\le t_2\le T$, it follow from \eqref{2.2}, \eqref{2.11}, \eqref{2.13} and \eqref{3.5} that
\begin{equation*}
	\begin{aligned}
	|b(t_2)-b(t)_1|
	&\le\frac{1}{2\mu+\lambda}(t_2-t_1)+C\int_{t_1}^{t_2}\|F\|^4_{L^{\infty}}dt\\
	&\le\frac{1}{2\mu+\lambda}(t_2-t_1)+C\int_{\sigma(T)}^{T}\|F\|_{L^6}^2\|\nabla F\|_{L^6}^2dt\\
	&\le\frac{1}{2\mu+\lambda}(t_2-t_1)+C\int_{\sigma(T)}^{T}\|(\rho+m)^{\frac{1}{2}}\dot{u}\|_{L^2}^2\|\nabla\dot{u}\|_{L^2}^2dt\\
	&\le\frac{1}{2\mu+\lambda}(t_2-t_1)+CC_0.
	\end{aligned}
\end{equation*}

Now, choosing $N_0=CC_0$, $N_1=\frac{1}{2\mu+\lambda}$ in \eqref{2.3} and setting $\hat{\zeta}=\bar{\rho}+\bar{m}$ in \eqref{2.4}, it gives that for all $\zeta\ge\hat{\zeta}=\bar{\rho}+\bar{m}$,
\begin{equation}\label{3.45}
	g(\zeta)=-\frac{\zeta}{2\mu+\lambda}\left(P(\zeta)-P_{\infty}\right)\le-\frac{\bar{\rho}+\bar{m}}{2\mu+\lambda}\le-\frac{1}{2\mu+\lambda}=-N_1,
\end{equation}
which together with Lemma \ref{Lm2.3}, \eqref{3.43} and \eqref{3.45} implies
\begin{equation}\label{3.46}
	\mathop{\rm sup}_{t\in[\sigma(T),T]}\|\rho+m\|_{L^\infty}\le\frac{3}{2}\left(\bar{\rho}+\bar{m}\right)+CC_0\le\frac{7}{4}\left(\bar{\rho}+\bar{m}\right),
\end{equation}
provided
\begin{equation}\label{3.47}
	C_0\le\varepsilon\triangleq{\rm min}\left\{\varepsilon_3,\frac{\bar{\rho}+\bar{m}}{4C}\right\}.
\end{equation}
The combination of \eqref{3.43} with \eqref{3.46} completes the proof of Lemma \ref{Lm3.8}.
\end{proof}
\section{A priori estimates~(\uppercase\expandafter{\romannumeral2}): higher order estimates}\label{S4}
Suppose $(\rho,m,u)$ is a smooth solution of \eqref{1.1}--\eqref{1.5}. In order to extend the classical solution globally in time, assume \eqref{3.47} holds, and the positive constant C may depend on
\begin{align*}
	T,\|g\|_{L^2},\|\nabla u_0\|_{H^1},\|\rho_0-\rho_\infty\|_{H^2\cap W^{2,q}},\|m_0-m_\infty\|_{H^2\cap W^{2,q}},\|P(\rho_0,m_0)-P_\infty\|_{H^2\cap W^{2,q}},
\end{align*}
for besides $\lambda$, $\mu$, $\gamma$, $\alpha$, M, $\Omega$, $M$, $\bar{\rho}$ and $\bar{m}$, where $g\in L^2(\Omega)$ is gives as in \eqref{1.10}, we give some necessary higher order estimates.
\begin{lemma}\label{Lm4.1}
There exists a positive constant C, such that
\begin{align}\label{4.1}
&\mathop{{\rm sup}}_{0\le t\le T}\|(\rho+m)^{\frac{1}{2}}\dot{u}\|_{L^2}+\int_{0}^{T}\|\nabla\dot{u}\|_{L^2}^2dt\le C \ \text{and}\\
\label{4.2}
&\mathop{{\rm sup}}_{0\le t\le T}\left(\|\nabla\rho\|_{L^2\cap L^6}+\|\nabla m\|_{L^2\cap L^6}+\|\nabla u\|_{H^1}\right)+\int_{0}^{T}(\|\nabla u\|_{L^\infty}+\|\nabla^2u\|_{L^6})dt\le C.
\end{align}
\end{lemma}
\begin{proof}
	By \eqref{3.18}, \eqref{3.36}, \eqref{3.37} and Lemma \ref{Lm2.6}, it gives
	\begin{equation*}
		\|\nabla u\|_{L^2}^2+\int_{0}^{T}\|(\rho+m)^{\frac{1}{2}}\dot{u}\|_{L^2}^2dt\le\int_{0}^{T}\|\nabla u\|_{L^2}^4dt+C,
	\end{equation*}
which together with Growall's inequality yields that
\begin{equation}\label{4.3}
	\mathop{{\rm sup}}_{0\le t\le T}\|\nabla u\|_{L^2}^2+\int_{0}^{T}\|(\rho+m)^{\frac{1}{2}}\dot{u}\|_{L^2}^2dt\le C.
\end{equation}
Choosing $h=0$ in \eqref{3.26},
we deduce from \eqref{2.11}, \eqref{2.15} and \eqref{4.3} that
\begin{equation*}
	\begin{aligned}
		&\mathop{\rm sup}_{0\le t\le T}\|(\rho+m)^{\frac{1}{2}}\dot{u}\|_{L^2}^2+\int_{0}^{T}\|\nabla\dot{u}\|_{L^2}^2dt\\
		&\le C\int_{\partial\Omega}|Fu\cdot\nabla n\cdot u|ds-\int_{\partial\Omega}|F_0u_0\cdot\nabla n\cdot u_0|ds+C\int_{0}^{T}(\|\nabla u\|_{L^2}^2+\|\nabla u\|_{L^2}^6+\|\nabla u\|_{L^4}^4)dt\\
		&\quad+\int_{0}^{T}\|(\rho+m)^{\frac{1}{2}}\dot{u}\|_{L^2}^2(1+\|\nabla u\|_{L^2}^4)dt\\
		&\le C\|\nabla F\|_{L^2}\|\nabla u\|^2_{L^2}+\|\nabla F_0\|_{L^2}\|\nabla u_0\|_{L^2}^2+C\\
		&\quad+\left(\mathop{\rm sup}_{0\le t\le T}\|\nabla u\|_{L^2}^2\right)^{\frac{1}{2}}\left(\mathop{\rm sup}_{0\le t\le T}\|(\rho+m)^{\frac{1}{2}}\dot{u}\|_{L^2}^2\right)^{\frac{1}{2}}\int_{0}^{T}\|(\rho+m)^{\frac{1}{2}}\dot{u}\|_{L^2}^2dt\\
		&\le\delta\mathop{\rm sup}_{0\le t\le T}\|(\rho+m)^{\frac{1}{2}}\dot{u}\|_{L^2}^2+C.
	\end{aligned}
\end{equation*}
Then choosing $\delta$ small enough, it gives \eqref{4.1}. Observe that for $2\le p\le6$, it indicates that
\begin{align*}
	&\left(|\nabla(\rho+m)|^p\right)_t+{\rm div}\left(|\nabla(\rho+m)|^pu\right)+(p-1)|\nabla(\rho+m)|^p{\rm div}u\\
	&+p|\nabla(\rho+m)|^{p-2}\left(\nabla(\rho+m)\right)^{tr}\nabla u\left(\nabla(\rho+m)\right)+p(\rho+m)|\nabla(\rho+m)|^{p-2}\nabla(\rho+m)\cdot\nabla{\rm div}u=0.
\end{align*}
Integrating the above equality over $\Omega$ and using \eqref{2.11} imply that
\begin{equation}\label{4.4}
	\begin{aligned}
		\left(\|\nabla(\rho+m)\|_{L^p}\right)_t\le&C(1+\|\nabla u\|_{L^\infty})\|\nabla(\rho+m)\|_{L^p}+\|\nabla F\|_{L^p}\\
		\le&C(1+\|\nabla u\|_{L^\infty})\|\nabla(\rho+m)\|_{L^p}+C\|(\rho+m)\dot{u}\|_{L^p}.
	\end{aligned}
\end{equation}
Moreover, by Lemma \ref{Lm2.8}, \eqref{1.19}, \eqref{2.11} and \eqref{2.15}, for any $2\le p\le6$, we have that
\begin{equation}\label{4.5}
	\begin{aligned}
	\|\nabla^2u\|_{L^p}\le& C(\|{\rm div}u\|_{W^{1,p}}+\|{\rm curl}u\|_{W^{1,p}}+\|\nabla u\|_{L^2})\\
	\le&C(\|(\rho+m)\dot{u}\|_{L^p}+\|\nabla P\|_{L^p}+\|(\rho+m)\dot{u}\|_{L^2}+\|\nabla u\|_{L^2}+\|P-P_\infty\|_{L^6}).
	\end{aligned}
\end{equation}

Next, it follows from \eqref{2.2}, \eqref{1.19}, \eqref{2.11}, \eqref{2.12}, \eqref{3.13} and \eqref{4.1} that
\begin{equation}\label{4.6}
	\begin{aligned}
		\|{\rm div}u\|_{L^\infty}+\|{\rm curl}u\|_{L^\infty}\le&C(\|F\|_{L^\infty}+\|P-{P}_\infty\|_{L^\infty})+\|{\rm curl}u\|_{L^\infty}\\
		\le&C\left(\|F\|_{L^2}+\|\nabla F\|_{L^6}+\|{\rm curl}u\|_{L^2}+\|\nabla{\rm curl}u\|_{L^6}+1\right)\\
		\le&C\left(\|(\rho+m)\dot{u}\|_{L^6}+\|P-{P}_\infty\|_{L^2}+\|\nabla u\|_{L^2}+\|(\rho+m)\dot{u}\|_{L^2}+1\right)\\
		\le&C(\|\nabla\dot{u}\|_{L^2}+1).
	\end{aligned}
\end{equation}
By Lemma \ref{Lm2.9}, \eqref{4.5} and \eqref{4.6}, we get
\begin{equation}\label{4.7}
	\begin{aligned}
	\|\nabla u\|_{L^\infty}\le&C\left(\|{\rm div}u\|_{L^\infty}+\|{\rm curl}u\|_{L^\infty}\right){\rm ln}(e+\|\nabla^2u\|_{L^6})+C(\|\nabla u\|_{L^2}+1)\\
	\le&C(\|\nabla\dot{u}\|_{L^2}+1){\rm ln}(e+\|\nabla^2u\|_{L^6})+C(\|\nabla u\|_{L^2}+1)\\
	\le&C(\|\nabla\dot{u}\|_{L^2}+1){\rm ln}\left(e+\|(\rho+m)\dot{u}\|_{L^6}+\|\nabla P\|_{L^6}+\|\nabla u\|_{L^2}\right)+C(\|\nabla u\|_{L^2}+1)\\
	\le&C(\|\nabla\dot{u}\|_{L^2}+1)({\rm ln}(e+\|\nabla\dot{u}\|_{L^2})+{\rm ln}(e+\|\nabla(\rho+m)\|_{L^6}))+C\\
	\le&C\left(\|\nabla\dot{u}\|^2_{L^2}+1\right)+C\left(\|\nabla\dot{u}\|_{L^2}+1\right){\rm ln}\left(e+\|\nabla(\rho+m)\|_{L^6}\right).
	\end{aligned}
\end{equation}
Combining \eqref{4.7} with \eqref{4.4} yields
\begin{align*}
	\left({\rm ln}\left(e+\|\nabla(\rho+m)\|_{L^6}\right)\right)_t\le C\left(1+(\|\nabla\dot{u}\|^2_{L^2}+1){\rm ln}(e+\|\nabla(\rho+m)\|_{L^6})\right)+C(\|\nabla\dot{u}\|_{L^2}+1).
\end{align*}
And then by Gronwall's inequality and \eqref{4.1}, we obtain
\begin{equation}\label{4.8}
	\mathop{\rm{sup}}_{0\le t\le T}\|\nabla(\rho+m)\|_{L^6}\le C.
\end{equation}
Moreover, \eqref{4.7} and \eqref{4.8} imply that
\begin{equation}\label{4.9}
	\int_{0}^{T}\|\nabla u\|_{L^\infty}dt\le C.
\end{equation}
Using the above inequality, \eqref{4.4} and \eqref{4.9}, when $p=2$ yields that
\begin{equation*}
		\mathop{\rm{sup}}_{0\le t\le T}\|\nabla(\rho+m)\|_{L^2}\le C,
\end{equation*}
which together with \eqref{4.1}, \eqref{4.5} and \eqref{4.8} gives that
\begin{equation*}
		\mathop{\rm{sup}}_{0\le t\le T}\|\nabla^2u\|_{L^2}\le C,\ \ \int_{0}^{T}\|\nabla^2u\|_{L^6}dt\le C.
\end{equation*}
Hence, we finish the proof of Lemma \ref{Lm4.1}.	
\end{proof}
\begin{lemma}\label{Lm4.2}
There exists a constant C such that
\begin{equation}\label{4.10}
	\mathop{\rm{sup}}_{0\le t\le T}\|(\rho+m)^{\frac{1}{2}}u_t\|^2_{L^2}+\int_{0}^{T}\|\nabla u_t\|_{L^2}^2dt\le C,
\end{equation}
\begin{equation}\label{4.11}
	\mathop{\rm{sup}}_{0\le t\le T}\left(\|\rho-{\rho}_\infty\|_{H^2}+\|m-{m}_\infty\|_{H^2}+\|P-{P}_\infty\|_{H^2}\right)\le C.
\end{equation}
\end{lemma}
\begin{proof}
	By Lemma \ref{Lm4.1}, a simple computation shows that
	\begin{equation*}
		\begin{aligned}
			\|(\rho+m)^{\frac{1}{2}}u_t\|_{L^2}^2\le&\|(\rho+m)^{\frac{1}{2}}\dot{u}\|_{L^2}^2+\|(\rho+m)^{\frac{1}{2}}u\cdot\nabla u\|_{L^2}^2\\
			\le&C+C\|(\rho+m)^\frac{1}{2}u\|_{L^3}^2\|\nabla u\|_{L^6}^2\\
			\le&C+C\|(\rho+m)^{\frac{1}{2}}u\|_{L^2}\|u\|_{L^6}\|\nabla u\|^2_{L^6}\\
			\le&C,
		\end{aligned}
	\end{equation*}
and
\begin{equation*}
	\begin{aligned}
	\int_{0}^{T}\|\nabla u_t\|_{L^2}^2dt\le&\int_{0}^{T}\|\nabla\dot{u}\|^2_{L^2}dt+\int_{0}^{T}\|\nabla(u\cdot\nabla u)\|^2_{L^2}dt\\
	\le&C+\int_{0}^{T}\left(\|\nabla u\|_{L^4}^4+\|u\|^2_{L^\infty}\|\nabla^2u\|^2_{L^2}\right)dt\\
	\le&C+\int_{0}^{T}\left(\|\nabla u\|_{L^2}\|\nabla u\|_{L^6}^3+\|\nabla u\|_{H^1}^2\right)dt\\
	\le&C,
	\end{aligned}
\end{equation*}
so we have \eqref{4.10}.
By using $\eqref{1.1}_1$, $\eqref{1.2}_2$ and \eqref{4.2}, it shows that
\begin{equation}\label{4.12}
\begin{aligned}
	\left(\|\nabla^2(\rho+m)\|_{L^2}\right)_t&\le C(1+\|\nabla^2u\|_{L^6}+\|\nabla u\|_{L^\infty})\|\nabla^2(\rho+m)\|_{L^2}+C\|\nabla^3u\|_{L^2}\\
	&\le C(1+\|\nabla^2u\|_{L^6}+\|\nabla u\|_{L^\infty})\|\nabla^2(\rho+m)\|_{L^2}+C(\|\nabla\dot{u}\|^2_{L^2}+1),
\end{aligned}
\end{equation}
where in the last inequality we have used the fact that
\begin{equation}\label{4.13}
	\begin{aligned}
	\|\nabla^3u\|_{L^p}&\le C(\|{\rm div}u\|_{W^{2,p}}+\|{\rm curl}u\|_{W^{2,p}}+\|\nabla u\|_{L^2})\\
	&\le C(\|(\rho+m)\dot{u}\|_{W^{1,p}}+\|P-P_\infty\|_{W^{2,p}}+\|\nabla u\|_{L^2}+\|(\rho+m)\dot{u}\|_{L^2}+\|P-P_\infty\|_{L^2}),
\end{aligned}
\end{equation}
for any $p\in[2,6]$  by \eqref{2.17}--\eqref{2.20} and \eqref{1.19}.

Employing Gronwall's inequality, \eqref{4.1}, \eqref{4.2}, and \eqref{4.12} leads to
\begin{equation*}
	\mathop{\rm{sup}}_{0\le t\le T}\|\nabla^2(\rho+m)\|_{L^2}\le C.
\end{equation*}
Hence,
\begin{equation}\label{4.14}
\|\nabla^2P\|_{L^2}\le C\|\nabla^2(\rho+m)\|_{L^2}\le C.
\end{equation}
Therefore, the proof of Lemma \ref{4.2} is completed.
\end{proof}
\begin{lemma}\label{Lm4.3}
There exists a constant C such that
\begin{equation}\label{4.15}
	\mathop{\rm{sup}}_{0\le t\le T}(\|\rho_t\|_{H^1}+\|m_t\|_{H^1}+\|P_t\|_{H^1})+\int_{0}^{T}(\|\rho_{tt}\|^2_{L^2}+\|m_{tt}\|^2_{L^2}+\|P_{tt}\|^2_{L^2})dt\le C,
\end{equation}
\begin{equation}\label{4.16}
	\mathop{\rm{sup}}_{0\le t\le T}\sigma\|\nabla u_t\|^2_{L^2}+\int_{0}^{T}\sigma\|(\rho+m)^{\frac{1}{2}}u_{tt}\|^2_{L^2}dt\le C.
\end{equation}
\end{lemma}
\begin{proof}
	By using $\eqref{1.1}_1$, $\eqref{1.1}_2$, we have
	\begin{equation}\label{4.17}
		\left(\rho+m\right)_t+u\cdot\nabla(\rho+m)+(\rho+m)\rm{div}u=0,
	\end{equation}
which together with \eqref{4.2} and \eqref{4.11} gives
\begin{equation}\label{4.18}
	\begin{aligned}
	\|(\rho+m)_t\|_{L^2}\le C\|u\|_{L^\infty}\|\nabla(\rho+m)\|_{L^2}+C\|\nabla u\|_{L^2}
	\le C\|\nabla u\|_{H^1}+C\le C.
	\end{aligned}
\end{equation}
Combining \eqref{4.17} with \eqref{4.2}, \eqref{4.11} implies
\begin{equation}\label{4.19}
	\begin{aligned}
	\|\nabla(\rho_t+m_t)\|_{L^2}&\le C\|\nabla u\|_{L^4}\|\nabla(\rho+m)\|_{L^4}+C\|u\|_{L^\infty}\|\nabla^2(\rho+m)\|_{L^2}+C\|\nabla^2u\|_{L^2}\\
	&\le C\|\nabla u\|_{H^1}\|\nabla(\rho+m)\|_{H^1}+C\|\nabla u\|_{H^1}+C\\
	&\le C.
	\end{aligned}
\end{equation}
Due to the fact that $P_t+u\cdot\nabla P+\gamma\rho^\gamma{\rm div}u+\alpha m^{\alpha}{\rm div}u=0$, we have
\begin{equation}\label{4.20}
	\begin{aligned}
	\|P_t\|_{L^2}\le C\|u\|_{L^\infty}\|\nabla P\|_{L^2}+C\|\nabla u\|_{L^2}
	\le C\|\nabla u\|_{H^1}+C\le C,
	\end{aligned}
\end{equation}
and
\begin{equation}\label{4.21}
	\begin{aligned}
		\|\nabla P_t\|_{L^2}&\le C\|\nabla\rho\|_{L^4}\|\rho_t\|_{L^4}+C\|\nabla\rho_t\|_{L^2}+C\|\nabla m\|_{L^4}\|m_t\|_{L^4}+C\|\nabla m_t\|_{L^2}\\
		&\le C\|\nabla\rho\|_{H^1}\|\rho_t\|_{H^1}+C\|\nabla m\|_{H^1}\|m_t\|_{H^1}+C\\
		&\le C.
	\end{aligned}
\end{equation}
By applying \eqref{4.18}--\eqref{4.21}, we get
\begin{equation}\label{4.22}
	\mathop{\rm{sup}}_{0\le t\le T}(\|\rho_t\|_{H^1}+\|m_t\|_{H^1}+\|P_t\|_{H^1})\le C.
\end{equation}
Differentiating $\eqref{1.1}_1$ and $\eqref{1.1}_2$ with respect to $t$ implies
\begin{equation}\label{4.23}
	(\rho+m)_{tt}+u_t\cdot\nabla(\rho+m)+u\cdot\nabla(\rho_t+m_t)+(\rho_t+m_t){\rm div}u+(\rho+m){\rm div}u_t=0.
\end{equation}
Combining \eqref{4.23} with \eqref{4.2}, \eqref{4.10} and \eqref{4.22} yields
\begin{equation*}
	\begin{aligned}
		\int_{0}^{T}\|(\rho+m)_{tt}\|_{L^2}^2dt&\le C\int_{0}^{T}\|u_t\|_{L^6}^2\|\nabla(\rho+m)\|_{L^3}^2dt+C\int_{0}^{T}\|u\|_{L^\infty}^2\|\nabla(\rho_t+m_t)\|^2_{L^2}dt\\
		&\quad+C\int_{0}^{T}\|\rho_t+m_t\|^2_{L^3}\|\nabla u\|^2_{L^6}dt+C\int_{0}^{T}\|\nabla u_t\|^2_{L^2}dt\\
		&\le C\int_{0}^{T}\|\nabla u_t\|^2_{L^2}dt+C\int_{0}^{T}\|\nabla^2u\|^2_{L^2}dt+C\\
		&\le C.
	\end{aligned}
\end{equation*}
Hence, it gives that
\begin{equation*}
	\begin{aligned}
		\int_{0}^{T}\|P_{tt}\|_{L^2}^2dt&\le C\int_{0}^{T}(\|\rho_t\|^4_{L^4}+\|\rho_{tt}\|_{L^2}^2+\|m_t\|^4_{L^4}+\|m_{tt}\|_{L^2}^2)dt\\
		&\le C\int_{0}^{T}(\|\rho_t\|^4_{H^1}+\|m_t\|^4_{H^1})dt+C\\
		&\le C.
	\end{aligned}
\end{equation*}
So we get \eqref{4.15}.

Next, differentiating $\eqref{1.1}_3$ with respect to $t$, and then multiplying by $u_{tt}$, yields that
\begin{equation}\label{4.24}
	\begin{aligned}
		&\left(\frac{2\mu+\lambda}{2}\|{\rm div}u_t\|_{L^2}^2+\frac{\mu}{2}\|{\rm curl}u_t\|^2_{L^2}\right)_t+\int(\rho+m)|u_{tt}|^2dx\\
		&=-\int(\rho+m)_tu_t\cdot u_{tt}dx-\int(\rho+m)_tu\cdot\nabla u\cdot u_{tt}dx-\int(\rho+m)u_t\cdot\nabla u\cdot u_{tt}dx\\
		&\quad -\int(\rho+m)u\cdot\nabla u_t\cdot u_{tt}dx-\int\nabla P_t\cdot u_{tt}dx\\
		&=:\sum_{i=1}^{5}I_i.
	\end{aligned}
\end{equation}
It follows from $\eqref{1.1}_1$, $\eqref{1.1}_2$, \eqref{4.2}, \eqref{4.10} and \eqref{4.15} that
\begin{equation}\label{4.25}
	\begin{aligned}
		I_1&=-\int(\rho+m)_tu_t\cdot u_{tt}dx\\
		&=-\frac{1}{2}\left(\int(\rho+m)_t|u_t|^2dx\right)_t+\frac{1}{2}\int(\rho+m)_{tt}|u_t|^2dx\\
		&=-\frac{1}{2}\left(\int(\rho+m)_t|u_t|^2dx\right)_t-\frac{1}{2}\int\left({\rm div}(\rho u+mu)\right)_t|u_t|^2dx\\
		&\le -\frac{1}{2}\left(\int(\rho+m)_t|u_t|^2dx\right)_t+C(\|(\rho+m)_t\|_{L^3}+\|(\rho+m)u_t\|_{L^3})\|\nabla u_t\|_{L^2}\|u_t\|_{L^6}\\
		&\le-\frac{1}{2}\left(\int(\rho+m)_t|u_t|^2dx\right)_t+C\|\nabla u_t\|^2_{L^2}(\|\nabla u_t\|^2_{L^2}+1),
	\end{aligned}
\end{equation}
	\begin{align}\label{4.26}
		I_2&=-\int(\rho+m)_tu\cdot\nabla u\cdot u_{tt}dx\notag\\
		&=-\left(\int(\rho+m)_tu\cdot\nabla u\cdot u_tdx\right)_t+\int(\rho+m)_{tt}u\cdot\nabla u\cdot u_tdx+\int(\rho+m)u_t\cdot\nabla u\cdot u_tdx\notag\\
		&\quad+\int(\rho+m)u\cdot\nabla u_t\cdot u_tdx\notag\\
		&\le-\left(\int(\rho+m)_tu\cdot\nabla u\cdot u_tdx\right)_t+C\|(\rho+m)_{tt}\|_{L^2}\|u_t\|_{L^6}\|u\|_{L^6}\|\nabla u\|_{L^6}\notag\\
		&\quad+C\|(\rho+m)^{\frac{1}{2}}u_t\|_{L^2}\|u_t\|_{L^6}\|\nabla u\|_{L^3}+C\|\nabla u_t\|_{L^2}\|u_t\|_{L^6}\|(\rho+m)^{\frac{1}{2}}u\|_{L^3}\notag\\
		&\le-\left(\int(\rho+m)_tu\cdot\nabla u\cdot u_tdx\right)_t+C(\|(\rho+m)_{tt}\|^2_{L^2}+\|\nabla u_t\|^2_{L^2}+1),
	\end{align}
\begin{equation}\label{4.27}
	\begin{aligned}
		I_3+I_4+I_5&=-\int(\rho+m)u_t\cdot\nabla u\cdot u_{tt}dx-\int(\rho+m)u\cdot\nabla u_t\cdot u_{tt}dx-\int\nabla P_t\cdot u_{tt}dx\\
		&\le C\|(\rho+m)^{\frac{1}{2}}u_{tt}\|_{L^2}\|u_t\|_{L^6}\|\nabla u\|_{L^3}+C\|(\rho+m)^{\frac{1}{2}}u_{tt}\|_{L^2}\|\nabla u_t\|_{L^2}\|u\|_{L^\infty}\\
		&\quad+\left(\int P_t{\rm div}u_tdx\right)_t-\int P_{tt}{\rm div}u_tdx\\
		&\le\delta\|(\rho+m)^{\frac{1}{2}}u_{tt}\|^2_{L^2}+\left(\int P_t{\rm div}u_tdx\right)_t+C(\|P_{tt}\|^2_{L^2}+\|\nabla u_t\|^2_{L^2}).
	\end{aligned}
\end{equation}
Choosing a suitably small positive constant $\delta$  and using \eqref{4.24}--\eqref{4.27}, we have
\begin{equation}\label{4.28}
	\begin{aligned}
	&\left(\frac{2\mu+\lambda}{2}\sigma\|{\rm div}u_t\|_{L^2}^2+\frac{\mu}{2}\sigma\|{\rm curl}u_t\|^2_{L^2}\right)_t+\sigma\int(\rho+m)|u_{tt}|^2dx\\
	&\le-\left(\frac{1}{2}\sigma\int(\rho+m)_t|u_t|^2dx+\sigma\int(\rho+m)_tu\cdot\nabla u\cdot u_tdx-\sigma\int P_t{\rm div}u_tdx\right)_t+C\sigma\|P_{tt}\|^2_{L^2}\\
	&\quad +C\sigma\|\nabla u_t\|^2_{L^2}(\|\nabla u_t\|^2_{L^2}+1)+C\sigma\|(\rho+m)_{tt}\|^2_{L^2}+C\|\nabla u_t\|^2_{L^2}+C.	
	\end{aligned}
\end{equation}
Integrating \eqref{4.28} over (0,T], using $\eqref{1.1}_1$, $\eqref{1.1}_2$, \eqref{4.10}, \eqref{4.15} and Lemma \ref{Lm2.6} gives
\begin{equation}\label{4.29}
	\begin{aligned}
		&\sigma\|\nabla u_t\|^2_{L^2}+\int_{0}^{T}\sigma\|(\rho+m)^{\frac{1}{2}}u_{tt}\|^2_{L^2}dt\\
		&\le-\frac{1}{2}\sigma\int(\rho+m)_t|u_t|^2dx-\sigma\int(\rho+m)_tu\cdot\nabla u\cdot u_tdx+\sigma\int P_t{\rm div}u_tdx\\
		&\quad+C\int_{0}^{T}\sigma\|\nabla u_t\|^2_{L^2}(\|\nabla u_t\|^2_{L^2}+1)dt+C\\
		&\le\frac{1}{2}\sigma\int{\rm div}\left((\rho+m)u\right)|u_t|^2dx+\delta\sigma\|\nabla u_t\|^2_{L^2}+C\int_{0}^{T}\sigma\|\nabla u_t\|^2_{L^2}(\|\nabla u_t\|^2_{L^2}+1)dt+C\\
		&\le C\sigma\|(\rho+m)^{\frac{1}{2}}u_t\|_{L^2}\|\nabla u_t\|_{L^2}+\delta\sigma\|\nabla u_t\|^2_{L^2}+C\int_{0}^{T}\sigma\|\nabla u_t\|^2_{L^2}(\|\nabla u_t\|^2_{L^2}+1)dt+C\\
		&\le\delta\sigma\|\nabla u_t\|^2_{L^2}+C\int_{0}^{T}\sigma\|\nabla u_t\|^2_{L^2}(\|\nabla u_t\|^2_{L^2}+1)dt+C.
	\end{aligned}
\end{equation}
By using \eqref{4.29}, \eqref{4.10} and Gronwall's inequality, we can obtain \eqref{4.16}.
\end{proof}
\begin{lemma}\label{Lm4.4}
For any $q\in(3,6)$, there exists a positive constant C such that
\begin{equation}\label{4.30}
	\mathop{\rm{sup}}_{0\le t\le T}\left(\|\rho-{\rho}_\infty\|_{W^{2,q}}+\|m-{m}_\infty\|_{W^{2,q}}+\|P-{P}_\infty\|_{W^{2,q}}\right)\le C,
\end{equation}
\begin{equation}\label{4.31}
	\mathop{\rm{sup}}_{0\le t\le T}\sigma\|\nabla u\|^2_{H^2}+\int_{0}^{T}\left(\|\nabla u\|^2_{H^2}+\|\nabla^2u\|_{W^{1,q}}^{p_0}+\sigma\|\nabla u_t\|^2_{H^1}\right)dt\le C,
\end{equation}
where $p_0=(1,\frac{9q-6}{10q-12})\in(1,\frac{7}{6})$.
\end{lemma}
\begin{proof}
	By \eqref{4.13}, \eqref{4.2} and \eqref{4.11}, it gives
	\begin{equation}\label{4.32}
		\begin{aligned}
			\|\nabla^2u\|_{H^1}\le&\|(\rho+m)\dot{u}\|_{H^1}+\|P-P_\infty\|_{H^2}+C\\
			\le&\|\nabla((\rho+m)\dot{u})\|_{L^2}+C\\
		\le&C\|\nabla u_t\|_{L^2}+C,		\end{aligned}
	\end{equation}
where we have used the fact that
\begin{equation*}
	\begin{aligned}
		\|\nabla((\rho+m)\dot{u})\|_{L^2}\le\|\nabla(\rho+m)\dot{u}\|_{L^2}+\|(\rho+m)\nabla\dot{u}\|_{L^2}
		\le C\|\nabla u_t\|_{L^2}+C.
	\end{aligned}
\end{equation*}
Then, we deduce from \eqref{4.2}, \eqref{4.10}, \eqref{4.16} and \eqref{4.32} that
\begin{equation}\label{4.33}
	\mathop{\rm{sup}}_{0\le t\le T}\sigma\|\nabla u\|^2_{H^2}+\int_{0}^{T}\|\nabla u\|^2_{H^2}dt\le C.
\end{equation}
Utilizing \eqref{4.2} and \eqref{4.15} implies that
\begin{equation}\label{4.34}
	\begin{aligned}
		\|\nabla u_t\|_{H^1}&\le C\left(\|\left((\rho+m)\dot{u}\right)_t\|_{L^2}+\|\nabla P_t\|_{L^2}+\|\nabla u_t\|_{L^2}\right)\\
		&\le C\left(\|(\rho+m)_t\dot{u}\|_{L^2}+\|(\rho+m)u_{tt}\|_{L^2}+\|(\rho+m)\left(u\cdot\nabla u\right)_t\|_{L^2}+\|\nabla u_t\|_{L^2}\right)+C\\
		&\le C\|(\rho+m)_t\|_{L^3}\|\nabla u_t\|_{L^2}+C\|(\rho+m)^{\frac{1}{2}}u_{tt}\|_{L^2}+C\|\nabla u_t\|_{L^2}+C\\
		&\le C\|\nabla u_t\|_{L^2}+C\|(\rho+m)^{\frac{1}{2}}u_{tt}\|_{L^2}+C,
	\end{aligned}
\end{equation}
where in the first inequality, we have used the a priori estimate similar to \eqref{4.5} since
\begin{equation*}
	\left\{
	\begin{array}{lr}
		\mu\Delta u_t+(\lambda+\mu)\nabla{\rm div}u_t=\left((\rho+m)\dot{u}\right)_t+\nabla P_t, & \ \ x\in\Omega, \\
		u_t\cdot n=0,\ \ {\rm curl}u_t
		\times n=0, & \ \ x\in\partial\Omega.
		\end{array}
	\right.
\end{equation*}
Combining \eqref{4.34} with \eqref{4.16} implies
\begin{equation}\label{4.35}
	\int_{0}^{T}\sigma\|\nabla u_t\|^2_{H^1}dt\le C.
\end{equation}
It follows from \eqref{4.13}, \eqref{4.1} and \eqref{4.11} that
\begin{equation}\label{4.36}
	\begin{aligned}
		\|\nabla^2u\|_{W^{1,q}}&\le C(\|(\rho+m)\dot{u}\|_{W^{1,q}}+\|\nabla P\|_{W^{1,q}}+\|\nabla u\|_{L^2}+\|P-P_\infty\|_{L^2}+\|P-P_{\infty}\|_{L^q})\\
		&\le C(\|(\rho+m)\dot{u}\|_{L^q}+\|\nabla((\rho+m)\dot{u})\|_{L^q}+\|\nabla P\|_{L^q}+\|\nabla^2 P\|_{L^q}+1)\\
		&\le C\|\nabla((\rho+m)\dot{u})\|_{L^q}+C\|\nabla^2P\|_{L^q}+C\|\nabla u_t\|_{L^2}+C,
	\end{aligned}
\end{equation}
which together with $\eqref{1.1}_1$ and $\eqref{1.1}_2$ gives
\begin{equation}\label{4.37}
	\begin{aligned}
		\left(\|\nabla^2(\rho+m)\|_{L^q}\right)_t&\le C(\|\nabla u\|_{L^\infty}+1)\|\nabla^2(\rho+m)\|_{L^q}+C\|\nabla^2u\|_{W^{1,q}}\\
		&\le C\left[(\|\nabla u\|_{L^\infty}+1)\|\nabla^2(\rho+m)\|_{L^q}+\|\nabla\left((\rho+m)\dot{u}\right)\|_{L^q}+\|\nabla u_t\|_{L^2}+1\right],
	\end{aligned}
\end{equation}
and
\begin{equation}\label{4.38}
	\begin{aligned}
		\|\nabla\left((\rho+m)\dot{u}\right)\|_{L^q}&\le C\|\nabla(\rho+m)\|_{L^{q}}\|u_t\|_{L^{\infty}}+C\|\nabla(\rho+m)\|_{L^{q}}\|u\|_{L^\infty}\|\nabla u\|_{L^{\infty}}\\
		&\quad+C\|\nabla u_t\|_{L^{q}}+C\|\nabla^2 u\|_{L^q}+C\|\nabla u\|_{H^2}^2\\
		&\le C\|\nabla u_t\|_{L^2}+C\|\nabla u_t\|_{L^2}^{\frac{6-q}{2q}}\|\nabla u_t\|_{L^{6}}^{\frac{3q-6}{2q}}+C\|\nabla u\|_{H^2}^2+C\\
		&\le C\left[\|\nabla u_t\|_{L^2}+\left(\sigma\|\nabla u_t\|_{L^2}^2\right)^{\frac{6-q}{2q}}\left(\sigma\|\nabla u_t\|^2_{H^1}\right)^{\frac{3q-6}{4q}}{\sigma}^{-\frac{1}{2}}+\|\nabla u\|_{H^2}^2+1\right]\\
		&\le C\|\nabla u_t\|_{L^2}+C\left(\sigma\|\nabla u_t\|^2_{H^1}\right)^{\frac{3q-6}{2q}}{\sigma}^{-\frac{1}{2}}+C\|\nabla u\|_{H^2}^2+C.
	\end{aligned}
\end{equation}
Hence, integrating inequality \eqref{4.38} over [0,T], by \eqref{4.1} and \eqref{4.35}, we obtain
\begin{equation}\label{4.39}
	\int_{0}^{T}\|\nabla\left((\rho+m)\dot{u}\right)\|_{L^q}^{p_0}dt\le C.
\end{equation}
Applying Gronwall's inequality to \eqref{4.37}, we deduce from \eqref{4.2} and \eqref{4.39} that
\begin{equation}\label{4.40}
	\mathop{\rm{sup}}_{0\le t\le T}\|\nabla^2(\rho+m)\|_{L^q}\le C,
\end{equation}
and then
\begin{equation*}
	\mathop{\rm{sup}}_{0\le t\le T}\left(\|P-{P}_\infty\|_{W^{2,q}}+\|\rho-{\rho}_\infty\|_{W^{2,q}}+\|m-{m}_\infty\|_{W^{2,q}}\right)\le C.
\end{equation*}
It follows from \eqref{4.36}, \eqref{4.39}, \eqref{4.40} and \eqref{4.10} that
\begin{equation*}
	\int_{0}^{T}\|\nabla^2u\|_{W^{1,q}}^{p_0}dt\le C.
\end{equation*}
So we finish the proof of Lemma \ref{Lm4.4}.
\end{proof}
\begin{lemma}\label{Lm4.5}
There exists a positive constant C such that
\begin{equation}\label{4.41}
	\mathop{\rm{sup}}_{0\le t\le T}\sigma^2\left(\|\nabla u_t\|^2_{H^1}+\|\nabla u\|^2_{W^{2,q}}\right)+\int_{0}^{T}\sigma^2\|\nabla u_{tt}\|^2_{L^2}dt\le C,
\end{equation}
for any $q\in(3,6)$.
\end{lemma}
\begin{proof}
	Differentiating $\eqref{1.1}_3$ with respect to $t$ twice gives
\begin{equation}\label{4.42}
	\begin{aligned}
		&(\rho+m)_{tt}u_t+2(\rho+m)_tu_{tt}+(\rho+m)u_{ttt}+\left[(\rho+m)_tu\cdot\nabla u+(\rho+m)(u_t\cdot\nabla u+u\cdot\nabla u_t)\right]_t\\
		&-(2\mu+\lambda)\nabla{\rm div}u_{tt}+\mu\nabla\times{\rm curl}u_{tt}+\nabla P_{tt}=0.
	\end{aligned}
\end{equation}
Then, multiplying \eqref{4.42} by $u_{tt}$ and integrating over $\Omega$, we conclude that
\begin{equation}\label{4.43}
	\begin{aligned}
		&\frac{1}{2}\left(\int(\rho+m)u_{tt}^2dx\right)_t+(2\mu+\lambda)\int({\rm div}u_{tt})^2dx+\mu\int({\rm curl}u_{tt})^2dx\\
		&=-\int(\rho+m)_{tt}u_t\cdot u_{tt}dx-\frac{3}{2}\int(\rho+m)_tu_{tt}^2dx-\int\nabla P_{tt}\cdot u_{tt}dx\\
		&\quad-\int\left((\rho+m)_tu\cdot\nabla u+(\rho+m)u_t\cdot\nabla u+(\rho+m)u\cdot\nabla u_t\right)_t\cdot u_{tt}dx\\
		&=:\sum_{i=1}^{4}J_i.
	\end{aligned}
\end{equation}
Now, we estimate all terms on the right-hand side of \eqref{4.43}. First, by \eqref{4.2}, \eqref{4.10} and \eqref{4.15}, it holds
\begin{equation}\label{4.44}
	\begin{aligned}
		&J_1+J_2+J_3\\
		&=-\int(\rho+m)_{tt}u_t\cdot u_{tt}dx-\frac{3}{2}\int(\rho+m)_t|u_{tt}|^2dx-\int\nabla P_{tt}\cdot u_{tt}dx\\
		&=\int{\rm div}((\rho+m)u)_tu_t\cdot u_{tt}dx+\frac{3}{2}\int{\rm div}((\rho+m)u)|u_{tt}|^2dx+\int P_{tt}{\rm div}u_{tt}dx\\
		&\le-\int((\rho+m)u)_t\cdot\nabla u_t\cdot u_{tt}dx-\int((\rho+m)u)_t\cdot\nabla u_{tt}\cdot u_{t}dx-\frac{3}{2}\int((\rho+m)u)\cdot\nabla u_{tt}\cdot u_{tt}dx\\
		&\quad+C\|P_{tt}\|^2_{L^2}+\delta\|\nabla u_{tt}\|^2_{L^2}\\
		&\le C\|(\rho+m)_t\|_{L^6}\|u\|_{L^6}(\|\nabla u_t\|_{L^2}\|u_{tt}\|_{L^6}+\|\nabla u_{tt}\|_{L^2}\|u_t\|_{L^6})+\|\nabla u_{tt}\|_{L^2}\|(\rho+m)u_{tt}\|_{L^2}\|u\|_{L^\infty}\\
		&\quad+C(\|\nabla u_t\|_{L^2}\|u_{tt}\|_{L^6}+\|\nabla u_{tt}\|_{L^2}\|u_t\|_{L^6})\|(\rho+m)u_t\|_{L^3}+C\|P_{tt}\|^2_{L^2}+\delta\|\nabla u_{tt}\|^2_{L^2}\\
		&\le C(\|\nabla u_t\|^2_{L^2}+\|(\rho+m)^{\frac{1}{2}}u_{tt}\|_{L^2}^2+\|\nabla u_t\|_{L^2}^4+\|P_{tt}\|_{L^2}^2)+\delta\|\nabla u_{tt}\|^2_{L^2},
	\end{aligned}
\end{equation}
and
\begin{align}\label{4.45}
		J_4&=-\int\left((\rho+m)_tu\cdot\nabla u+(\rho+m)u_t\cdot\nabla u+(\rho+m)u\cdot\nabla u_t\right)_t\cdot u_{tt}dx\notag\\
		&=-\int(\rho+m)_{tt}u\cdot\nabla u\cdot u_{tt}dx-2\int(\rho+m)_tu_t\cdot\nabla u\cdot u_{tt}dx-2\int(\rho+m)_tu\cdot\nabla u_t\cdot u_{tt}dx\notag\\
		&\quad-\int(\rho+m)u_{tt}\cdot\nabla u\cdot u_{tt}dx-2\int(\rho+m)u_t\cdot\nabla u_t\cdot u_{tt}dx-\int(\rho+m)u\cdot\nabla u_{tt}\cdot u_{tt}dx\notag\\
		&\le C\|(\rho+m)_{tt}\|_{L^2}\|\nabla u\|_{L^3}\|u_{tt}\|_{L^6}\|u\|_{L^\infty}+C\|(\rho+m)_t\|_{L^3}\|u_t\|_{L^6}\|\nabla u\|_{L^3}\|u_{tt}\|_{L^6}\notag\\
		&\quad+C\|(\rho+m)_t\|_{L^3}\|\nabla u_t\|_{L^2}\|u_{tt}\|_{L^6}\|u\|_{L^\infty}+C\|(\rho+m)^{\frac{1}{2}}u_{tt}\|_{L^2}\|\nabla u\|_{L^3}\|u_{tt}\|_{L^6}\notag\\
		&\quad+C\|(\rho+m)u_t\|_{L^3}\|\nabla u_t\|_{L^2}\|u_{tt}\|_{L^6}+C\|\nabla u_{tt}\|_{L^2}\|(\rho+m)^{\frac{1}{2}}u_{tt}\|_{L^2}\|u\|_{L^\infty}\notag\\
		&\le\delta\|\nabla u_{tt}\|^2_{L^2}+C\|(\rho+m)_{tt}\|^2_{L^2}+C\|(\rho+m)^{\frac{1}{2}}u_{tt}\|^2_{L^2}+C\|\nabla u_t\|^2_{L^2}(\|\nabla u_t\|^2_{L^2}+1).
\end{align}
Due to the fact that
\begin{equation}\label{4.46}
	\|\nabla u_{tt}\|_{L^2}\le C\left(\|{\rm div}u_{tt}\|_{L^2}+\|{\rm curl}u_{tt}\|_{L^2}\right)
\end{equation}
Using \eqref{4.43}--\eqref{4.46} and  choosing enough small $\delta$, we obtain
\begin{equation}\label{4.47}
	\begin{aligned}
	&\left(\int(\rho+m)|u_{tt}|^2dx\right)_t+\|\nabla u_{tt}\|^2_{L^2}\\
	&\le C\|(\rho+m)_{tt}\|^2_{L^2}+C\|(\rho+m)^{\frac{1}{2}}u_{tt}\|^2_{L^2}+C\|P_{tt}\|^2_{L^2}+C\|\nabla u_t\|^2_{L^2}(\|\nabla u_t\|^2_{L^2}+1),
	\end{aligned}
\end{equation}
which together with \eqref{4.10}, \eqref{4.15} and \eqref{4.16} gives that
\begin{equation}\label{4.48}
	\mathop{\rm{sup}}_{0\le t\le T}\sigma^2\|(\rho+m)^{\frac{1}{2}}u_{tt}\|^2_{L^2}+\int_{0}^{T}\sigma^2\|\nabla u_{tt}\|^2_{L^2}dt\le C.
\end{equation}
Furthermore, it follows from \eqref{4.34}, \eqref{4.16} and \eqref{4.48} that
\begin{equation}\label{4.49}
	\mathop{\rm{sup}}_{0\le t\le T}\sigma^2\|\nabla u_t\|^2_{H^1}\le C.
\end{equation}
Finally, combining \eqref{4.36} with \eqref{4.38}, \eqref{4.16}, \eqref{4.30} and \eqref{4.33} that
\begin{equation}\label{4.50}
	\mathop{\rm{sup}}_{0\le t\le T}\sigma^2\|\nabla u\|^2_{W^{2,q}}\le C,
\end{equation}
which together with \eqref{4.48} and \eqref{4.49} gives \eqref{4.41} and this completes the proof of Lemma \ref{Lm4.5}.
\end{proof}
\section{Proofs of Theorems \ref{Thm1.1} and \ref{Thm1.2}}\label{S5}
With the priori proof in Section \ref{S3} and Section \ref{S4} at hand, we prove the main results of this paper in this section.

$\underline{\emph{Proof of Theorem \ref{Thm1.1}.}}$  By Lemma \ref{Lm2.1}, the problem \eqref{1.1}--\eqref{1.5} has a unique classical solution $(\rho,m,u)$ on $\Omega\times(0,T_{\ast}]$ for some $T_{\ast}>0$.  Now, we will
extend the classical solution $(\rho,m,u)$ globally in time.
\par Firstly, by \eqref{3.2} and \eqref{3.3}, it is easy to check that
\begin{equation*}
	A_1(0)+A_2(0)=0,\ \ 0\le\rho_0+m_0\le\bar{\rho}+\bar{m},\ \ A_3(0)\le M.
\end{equation*}
Then, there exists a $T_1\in(0,T_*]$ such that
\begin{equation}\label{5.1}
	0\le\rho_0+m_0\le2(\bar{\rho}+\bar{m}),\ \ A_1(T_1)+A_2(T_1)\le2C_0^{\frac{1}{2}},\ \ A_3(\sigma(T_1))\le2M.
\end{equation}
Set
\begin{equation}\label{5.2}
	T^*={\rm sup}\left\{T|\ (5.1)\text{ holds }\right\}.
\end{equation}
Clearly, $0<T_1\le T^*$. And for any $0<\tau<T\le T^*$, one deduces from Lemmas \ref{Lm4.3}--\ref{Lm4.5} that
\begin{equation}\label{5.3}
	\left\{
	\begin{array}{lr}
		\rho-{\rho}_{\infty}\in C\left([0,T];W^{2,q}\right), & \\
			m-{m}_\infty\in C\left([0,T];W^{2,q}\right), & \\
			\nabla u_t\in C\left([\tau,T];L^q\right), & \\
			\nabla u,\nabla^2u\in C\left([\tau,T];C(\bar{\Omega})\right), &
			\end{array}
\right.
\end{equation}
where one has taken advantage of the standard embedding:
\begin{equation*}
	L^{\infty}(\tau,T;H^1)\cap H^1(\tau,T;H^{-1})\hookrightarrow C([\tau,T];L^q),\text{ for any } q\in[2,6).
\end{equation*}
This in particular yields
\begin{equation}\label{5.4}
	(\rho+m)^{\frac{1}{2}}u_t,(\rho+m)^{\frac{1}{2}}\dot{u}\in C([\tau,T];L^2).
\end{equation}
Next, we claim that
\begin{equation}\label{5.5}
	T^*=\infty.
\end{equation}
Otherwise, $T^*<\infty$. By Proposition \ref{Prop3.1}, it holds that
\begin{equation}\label{5.6}
	0\le\rho+m\le\frac{7}{4}\left(\bar{\rho}+\bar{m}\right),\ \ A_1(T^*)+A_2(T^*)\le C_0^{\frac{1}{2}},\ \ A_3(\sigma(T^*))\le M.
\end{equation}
We deduce from Lemma \ref{Lm4.4}, Lemma \ref{Lm4.5} and \eqref{5.4} that $\left(\rho(x,T^*),m(x,T^*),u(x,T^*)\right)$ satisfy the initial date condition \eqref{1.7}--\eqref{1.10}, where $g(x)\triangleq(\rho+m)^{\frac{1}{2}}\dot{u}(x,T^*)$, $x\in\Omega$. Hence, Lemma \ref{Lm2.1} shows that there is a $T^{**}>T^*$, such that \eqref{5.1} holds for $T=T^{**}$, which contradicts the definition of $T^*$.

By Lemma \ref{Lm2.1}, Lemma \ref{Lm4.4}, Lemma \ref{Lm4.5} and \eqref{5.3} indicates that $(\rho,m,u)$ is the unique classical solution defined on $\Omega\times(0,T]$ for any $0<T<T^*=\infty$.

Finally, to finish the proof of Theorem \ref{Thm1.1}, it remains to prove \eqref{1.14}. It is easy to have
\begin{equation}\label{5.7}
	(P-{P}_\infty)_t+u\cdot\nabla P+\gamma\rho^{\gamma}{\rm div}u+\alpha m^{\alpha}{\rm div}u=0.
\end{equation}
Multiplying \eqref{5.7} by 4$(P-{P}_\infty)^3$,
one has
\begin{equation*}
	\left(\|P-{P}_\infty\|_{L^4}^4\right)_t\le C\|{\rm div}u\|_{L^2}^2+C\|P-{P}_\infty\|_{L^4}^4,
\end{equation*}
which together with \eqref{3.7} and \eqref{3.30} yields that
\begin{equation}\label{5.8}
	\int_{1}^{\infty}\left(\|P-{P}_\infty\|_{L^4}^4\right)_tdt\le C.
	\end{equation}
Combining \eqref{3.30} with \eqref{5.8} leads to
\begin{equation}\label{5.9}
	\mathop{{\rm lim}}_{t\to\infty}\|P-{P}_\infty\|_{L^4}^4=0.
\end{equation}
For $2< q<\infty$, by \eqref{5.9}, we get
\begin{equation}\label{5.10}
	\mathop{{\rm lim}}_{t\to\infty}\|P-{P}_\infty\|_{L^q}=0.
\end{equation}
Notice that \eqref{3.7} imply
\begin{equation}\label{5.11}
	\begin{aligned}
		\int(\rho+m)^{\frac{1}{2}}|u|^4dx&\le\left(\int(\rho+m)|u|^2dx\right)^{\frac{1}{2}}\|u\|^{3}_{L^{6}}\le C\|\nabla u\|^{3}_{L^2}.
	\end{aligned}
\end{equation}
Thus, \eqref{1.14} follows provided that
\begin{equation}\label{5.12}
		\mathop{{\rm lim}}_{t\to\infty}\|\nabla u\|_{L^2}=0.
\end{equation}

Choosing $h=0$ in \eqref{3.18} and integrating it over $(1,\infty)$, and using \eqref{2.15}, \eqref{3.5}, \eqref{3.7} and \eqref{3.30}, we get
\begin{equation}\label{5.13}
	\begin{aligned}
	&\int_{1}^{\infty}|\phi'(t)|^2dt\\
    &\le C\int_{1}^{\infty}(\|\nabla u\|_{L^2}^2+\|\nabla u\|_{L^2}^4+\|\nabla u\|_{L^3}^3)dt+C\|\nabla u\|_{L^2}^2+C\|P-P_\infty\|_{L^2}^2\\
	&\le C\int_{1}^{\infty}(\|\nabla u\|_{L^2}^2+\|\nabla u\|_{L^2}^4+\|\nabla u\|_{L^2}^3+\|\nabla u\|_{L^2}^{\frac{3}{2}}\|P-{P}_{\infty}\|_{L^4}+\|(\rho+m)\dot{u}\|_{L^2}^3)dt+C\\
	&\le C,
	\end{aligned}
\end{equation}
where $\phi(t)=\frac{\lambda+2\mu}{2}\|{\rm div}u\|_{L^2}^2+\frac{\mu}{2}\|{\rm curl}u\|_{L^2}^2$.
By \eqref{3.7}, we obtain that
\begin{equation*}
	\int_{1}^{\infty}\|\nabla u\|^2_{L^2}dt\le\int_{0}^{\infty}\|\nabla u\|_{L^2}^2dt\le C,
\end{equation*}
which together with \eqref{5.13} yields \eqref{5.12}.

$\underline{\emph{Proof of Theorem \ref{Thm1.2}.}}$
Now, we will prove Theorem \ref{Thm1.2} by contradiction. Suppose that there exist some constant $C_1>0$ and a subsequence $\left\{t_{n_j}\right\}_{j=1}^{\infty}$ with $t_{n_j}\to\infty$ as $j\to\infty$, such that $\|\nabla P(\cdot,t_{n_j})\|_{L^r}\le C_1$. Thanks to \eqref{2.2}, for $a=3r/\left(3r+4(r-3)\right)\in(0,1)$, 
it holds that
\begin{equation}\label{5.14}
	\begin{aligned}
		\|P(x,t_{n_j})-{P}_\infty\|_{C(\bar{\Omega})}&\le C\|\nabla P(x,t_{n_j})\|_{L^r}^a\|P(x,t_{n_j})-{P}_\infty\|_{L^4}^{1-a}\\
		&\le CC_1^a\|P(x,t_{n_j})-{P}_\infty\|_{L^4}^{1-a},
	\end{aligned}
\end{equation}
which together with \eqref{1.14} yields that
\begin{equation}\label{5.15}
	\|P(x,t_{n_j})-{P}_\infty\|_{C(\bar{\Omega})}\to0\ \ as\ \  t_{n_j}\to\infty.
\end{equation}

On the other hand, since $(\rho,m,u)$ is a classical solution satisfying \eqref{1.1}, there exists a unique particle path $x_0(t)$ with $x_0(t)=x_0$ such that
\begin{equation*}
	P(x_0(t),t)\equiv0\  \text{ for all }\ t>0.
\end{equation*}
Hence, we have
\begin{equation*}
	\|P(x,t_{n_j})-{P}_\infty\|_{C(\bar{\Omega})}\ge|P(x_0(t_{n_j}),t_{n_j})-{P}_\infty|\equiv P_\infty>0,
\end{equation*}
which contradicts \eqref{5.15}. So we get the desired result \eqref{1.15}. Thus, we finish the proof of Theorem \ref{Thm1.2}.
\section*{Acknowledgement}
Z. Li is supported by the NSFC (No. 12126316, No. 11931013) and Innovative Research Team of Henan Polytechnic University (No. T2022-7). H. Wang is supported by the National Natural Science Foundation of China (No. 11901066), the Natural Science Foundation of Chongqing (No. cstc2019jcyj-msxmX0167), and projects Nos. 2022CDJXY-001, 2020CDJQY-A040
supported by the Fundamental Research Funds for the Central Universities.

\bibliographystyle{plain}	
\end {document}